\documentclass[a4paper,12pt,draft]{article}

\usepackage{amsthm}
\usepackage{latexsym}
\usepackage{dsfont}
\usepackage{bbm}
\usepackage{amssymb}
\usepackage{amsmath}
\usepackage{graphicx}

\numberwithin{equation}{section}

\theoremstyle{plain}
\newtheorem{Thm}{Theorem}[section]
\newtheorem{Lemma}[Thm]{Lemma}

\newtheorem{Prop}[Thm]{Proposition}
\theoremstyle{remark}
\newtheorem{Rem}[Thm]{Remark}
\theoremstyle{definition}

\newtheorem{Con}[Thm]{Condition}

\newcommand{\N}{\mathbb{N}}
\newcommand{\integers}{\mathbb{Z}}

\newcommand{\R}{\mathbb{R}}

\newcommand{\Prob}{\mathbb{P}}
\newcommand{\Prm}{\mathrm{P}}
\DeclareMathOperator{\Var}{\mathrm{Var}}
\newcommand{\E}{\mathbb{E}}

\newcommand{\A}{\mathcal{A}}
\newcommand{\B}{\mathcal{B}}
\newcommand{\I}{\mathcal{I}}
\newcommand{\J}{\mathcal{J}}
\newcommand{\F}{\mathcal{F}}
\newcommand{\G}{\mathcal{G}}

\renewcommand{\H}{\mathcal{H}}
\renewcommand{\L}{\mathcal{L}}

\newcommand{\Z}{\mathcal{Z}}
\DeclareMathOperator{\iZ}{^{\mathit{i}}\!\mathcal{Z}}
\DeclareMathOperator{\jZ}{^{\mathit{j}}\!\mathcal{Z}}
\DeclareMathOperator{\oneZ}{^{1}\!\mathcal{Z}}
\DeclareMathOperator{\iN}{^{\mathit{i}}\!\mathit{N}}
\DeclareMathOperator{\jN}{^{\mathit{j}}\!\mathit{N}}

\DeclareMathOperator{\iX}{^{\mathit{i}}\!\mathit{X}}
\DeclareMathOperator{\jX}{^{\mathit{j}}\!\mathit{X}}
\DeclareMathOperator{\iW}{^{\mathit{i}}\!\mathit{W}}

\DeclareMathOperator{\iWone}{^{\mathit{i}}\!\mathit{W}_{1}}
\DeclareMathOperator{\iWn}{^{\mathit{i}}\!\mathit{W}_{\mathit{n}}}
\DeclareMathOperator{\jWn}{^{\mathit{j}}\!\mathit{W}_{\mathit{n}}}
\DeclareMathOperator{\oneW}{^{1}\!\mathit{W}}
\DeclareMathOperator{\itau}{^{\mathit{i}}\!\tau}
\DeclareMathOperator{\jtau}{^{\mathit{j}}\!\tau}

\DeclareMathOperator{\imu}{^{\mathit{i}}\!\mu}
\DeclareMathOperator{\jmu}{^{\mathit{j}}\!\mu}
\DeclareMathOperator{\onemu}{^{1}\!\mu}
\DeclareMathOperator{\im}{^{\mathit{i}}\!\mathit{m}}
\DeclareMathOperator{\jm}{^{\mathit{j}}\!\mathit{m}}
\DeclareMathOperator{\iphi}{^{\mathit{i}}\!\overline{\phi}}

\newcommand{\bfM}{\mathbf{M}}
\newcommand{\bS}{\mathbf{S}}
\newcommand{\bT}{\mathbf{T}}
\newcommand{\bu}{\mathbf{u}}
\newcommand{\bv}{\mathbf{v}}

\newcommand{\1}{\mathbbm{1}}
\newcommand{\transp}{\textsf{T}}

\newcommand{\comp}{\mathsf{c}}

\newcommand{\dy}{\mathrm{d}\mathit{y}}

\newcommand{\ds}{\mathrm{d}\mathit{s}}
\newcommand{\dt}{\mathrm{d}\mathit{t}}
\newcommand{\du}{\mathrm{d}\mathit{u}}

\renewcommand{\dj}{\mathrm{d}\mathit{j}}

\title{Rate of convergence in the law of large numbers for supercritical general multi-type branching processes}
\author{Alexander Iksanov and Matthias Meiners}
\begin{document}

\thispagestyle{empty}
\maketitle

\begin{abstract}
\vspace{0,1cm}
We provide sufficient conditions for polynomial rate of convergence in the weak law of large numbers
for supercritical general indecomposable multi-type branching processes.
The main result is derived
by investigating the embedded single-type process composed of all individuals having the same type as the ancestor.
As an important intermediate step, we determine the (exact) polynomial rate of convergence of
Nerman's martingale in continuous time to its limit.
The techniques used also allow us to give streamlined proofs of the weak and strong laws of large numbers
and ratio convergence for the processes in focus.

\vspace{0.05cm}
\noindent
\emph{Keywords:} Markov renewal theory $\cdot$ supercritical general multi-type branching process

\noindent
2010 Mathematics Subject Classification: Primary: 60J80 \\							
\hphantom{2010 Mathematics Subject Classification:} Secondary:  60K15			
\end{abstract}

\section{Introduction} \label{sec:Intro}

In the present paper, we derive sufficient conditions for polynomial rate of convergence in the weak law of large numbers
for supercritical general indecomposable multi-type branching processes.
As a by-product of our analysis, we give new proofs of the weak and strong laws of large numbers and ratio convergence for these processes
based on the corresponding results for single-type processes.

\subsection{Model description}  \label{subsec:model description}

Let $\N := \{1,2,\ldots\}$ denote the set of positive integers
and $\N^0 = \{\varnothing\}$ the set that contains the empty tuple only.
Define $\I := \bigcup_{n \geq 0} \N^n$ to be the set of finite tuples of positive integers.
Members of $\I$ are called (potential) individuals and are typically denoted by the letters $x,y,z$.
If $x=(x_1,\ldots,x_n)$, we write $|x|=n$ and call $n$ the generation of $x$.
If $y=(y_1,\ldots,y_m)$, then we write $xy$ for $(x_1,\ldots,x_n,y_1,\ldots,y_m)$.
$x|_k$ is defined as the `ancestor of $x$ in the $k$th generation', that is, $x|_k=(x_1,\ldots,x_k)$ if $k \leq |x|$.
$x \prec y$ means that $|x| < |y|$ and $y|_{|x|} = x$ while $x \preceq y$ means that either $x \prec y$ or $x = y$.
If $J \subseteq \I$ is a set of individuals, we write $x \prec J$ ($x \preceq J$) if $y \not \preceq x$ (resp., $y \not \prec x$) for all $y \in J$.
In words, $x \prec J$ if $x$ has no ancestor in $J$.

Let $(\Omega_{\varnothing}, \A_{\varnothing},
\Prm_{\!\varnothing})$ be a probability space on which point
processes $\iZ = \sum_{k=1}^{\iN} \delta_{(\itau_k,\iX_k)}$ on
$\{1,\ldots,p\} \times \R_{\geq 0}$, $i = 1,\ldots,p$ are
defined where here and in the remainder of the paper, $\delta_x$
denotes the Dirac measure with a point at $x$ and $\R_{\geq 0}:=[0,\infty)$.
Notice that $\Prm_{\!\varnothing}(\iN = \infty) > 0$ is not excluded. For $i=1,\ldots,p$, the basic
probability space is defined to be the product space
\begin{equation*}
(\Omega,\A,\Prob^i) ~:=~ (\{1,\ldots,p\},
\mathfrak{P}(\{1,\ldots,p\}), \delta_i) \otimes \prod_{x \in \I}
(\Omega_{x}, \A_{x}, \Prm_{\!x}),
\end{equation*}
where $(\Omega_{x}, \A_{x}, \Prm_{\!x})$, $x \in \I$ are copies of
$(\Omega_{\varnothing}, \A_{\varnothing}, \Prm_{\!\varnothing})$
and $\mathfrak{P}(\{1,\ldots,p\})$ is the set of all
subsets of $\{1,\ldots, p\}$. In particular, each space
$(\Omega_{x}, \A_{x}, \Prm_{\!x})$ carries copies
$\jZ(x)\!=\!\sum_{k=1}^{N^{j}(x)} \! \delta_{(\jtau_k(x),\jX_k(x))}$ of
the point processes $\jZ$, $j=1,\ldots,p$. We slightly abuse
notation and interpret the $(\oneZ(x),\ldots,\,^p\!\!\Z(x))$, $x
\in \I$ as i.i.d.~processes on $(\Omega,\A,\Prob^i)$.

We now describe the evolution of the process. Interpreting
$\varnothing$ as the label of the ancestor and $\tau(\varnothing)$
and $S(\varnothing)$ as its type and birth time, respectively, we
put $\tau(\varnothing)=i$ (under $\Prob^i$; formally, $\tau(\varnothing)$ is the projection onto the first coordinate of $\Omega$)
and $S(\varnothing)=0$.
At time $n=1$, the ancestor produces offspring according to the
point process $\iZ(\varnothing)$.
The offspring is enumerated by $1, \ldots, \iN =
\iZ(\varnothing)(\{1,\ldots,p\} \times \R_{\geq 0})$.
(Notice that when $\iN = \infty$, there is no individual labeled $\infty$. We make the convention that an enumeration of
the form $1,2,\ldots,\infty$ means the enumeration $1,2,\ldots$.)
Type and birth time of individual $x$ are defined by $\tau(x) :=
\itau_x(\varnothing)$ and $S(x) := \iX_x(\varnothing)$,
respectively. $\G_1 = \{1,2,\ldots, \iN\}$ is the first generation
of the process. Further, an individual $x = x_1 \ldots x_n \in
\N^n$ of the $n$th generation with type $\tau(x)=j$ and birth time
$S(x)$ produces at time $n+1$ a random number $\jN(x)$ offspring.
The offspring are labeled $x1, \ldots, x \jN(x)$. For $y
\in \N$, $y \leq \jN(x)$, type and birth time of particle $xy$ are
given by $\jtau_y(x)$ and $S(x) + \jX_y(x)$, respectively. The
$(n+1)$st generation $\G_{n+1}$ is defined by
\begin{equation}    \label{eq:G_n+1}
\{xy: x \in \G_n \text{ and } y \in \N,\, y \leq \,\!^{\tau(x)}\!N(x)\}.
\end{equation}
We set $\G := \bigcup_{n \in \N_0} \G_n$.
The point process of types and positions of the $n$th generation individuals will be denoted by
\begin{equation}    \label{eq:Z_n}
\Z_n    ~:=~    \sum_{|x| = n}  \delta_{(\tau(x),S(x))}
\end{equation}
where here and in what follows summation over $|x|=n$ means summation over $x \in \G_n$.
The sequence $(\Z_n)_{n \geq 0}$ forms a \emph{multi-type branching random walk}.

We further assume the existence of a product-measurable, separable random characteristic
$\phi: \Omega \times \R \to [0,\infty)$ with $\phi(t)=0$ for all $t<0$.
For $t \in \R$, we write $\phi(t)$ for the random variable $\omega \mapsto \phi(\omega,t)$.
Notice that $\phi$ may depend on the types of all individuals in $\I$, in particular on the type of the ancestor.

To define the general branching process counted with
characteristic $\phi$, we need to introduce further notation. An
element $\omega \in \Omega$ is of the form $\omega = (i,(\omega_x)_{x \in \I})$.
For each $x \in \I$, let $\sigma_x:\Omega \to \Omega$, $\omega = (\omega_y)_{y \in \I} \mapsto \sigma_x \omega := (\tau(x),(\omega_{xy})_{y \in \I})$
be the shift operator. Whenever $\Psi$ is a function from
$(\Omega,\A)$ into another measurable space, we denote by
$[\Psi]_x$ the function $\omega \mapsto \Psi(\sigma_x \omega)$.
The general (multi-type) branching process counted with
characteristic $\phi$ is then defined as
\begin{equation}    \label{eq:BGP}
\Z^{\phi}(t)    ~:=~    \sum_{x \in \G} [\phi]_x(t-S(x)).
\end{equation}

\section{Main results}  \label{sec:main results}

It is known from \cite[Theorem 6.5]{Nerman:1979}, \cite[Theorem 7.2]{Jagers:1989} and \cite[Theorem 2.1]{Olofsson:2009}
that under appropriate assumptions which include
the existence of a Malthusian parameter $\alpha > 0$, $e^{-\alpha t} \Z^{\phi}(t)$ converges in probability to a limit
which is not degenerate at $0$.\!\footnote{
Notice that in \cite{Jagers:1989,Olofsson:2009} the more general situation of an abstract type space is considered.}
The main result of the paper at hand is Theorem \ref{Thm:Rates of convergence},
in which sufficient conditions are provided for $e^{-\alpha t} \Z^{\phi}(t)$ to converge to its limit in probability at a polynomial rate.
Additionally, the methods employed here allow us to give simple proofs of the convergence in probability and a.s.~convergence
of $e^{-\alpha t} \Z^{\phi}(t)$ and the a.s.~convergence of $\Z^{\phi}(t)/\Z^{\psi}(t)$ which are stated as Theorems \ref{Thm:WLLN},
\ref{Thm:SLLN} and \ref{Thm:Ratio convergence}, respectively.
Theorem \ref{Thm:Ratio convergence} improves on an earlier result by Nerman \cite[Theorem 6.7]{Nerman:1979}.

\subsection{Preliminaries and assumptions}  \label{subsec:Assumptions}

For $i,j \in \{1,\ldots,p\}$, let $\imu(\dj \times \dt)$ denote
the intensity measure of the point process $\iZ(\dj \times \dt)$
on $\R_{\geq 0}$, that is, $\imu(\{j\} \times B) =
\E^i [\iZ(\{j\} \times B)]$ for Borel sets $B \subset
\R_{\geq 0}$. By $\im_{j}$ we denote the Laplace transform
of $\imu(\{j\} \times \cdot)$:
\begin{equation}    \label{eq:m^sigma_tau}
\im_{j}(\theta) ~:=~    \int_{[0,\infty)} e^{-\theta t} \imu(\{j\} \times \dt)
~=~ \E^{i} \bigg[\sum_{|x|=1: \tau(x)=j} e^{-\theta S(x)}\bigg],    \quad   \theta \geq 0.
\end{equation}
Let $\bfM(\theta)$ denote the matrix with entries $\bfM(\theta)_{ij} = \im_{j}(\theta)$, $i,j=1,\ldots, p$.
Each $\bfM(\theta)$ is a nonnegative matrix that may have entries $+\infty$.
Throughout the paper, we make the following assumptions:
\begin{itemize}
    \item[(A1)]
        For all $h > 0$ and all $h_1,\ldots,h_p \in [0,h)$
        $\imu(\{j\} \times (h_j - h_i + h\integers)^{\comp}) > 0$ for some $i,j \in \{1,\ldots,p\}$
        where $A^\comp$ denotes the complement of $A$ and $\integers$ is the set of integers.
        Further, $\bfM(0)$ is irreducible, \textit{i.e.}, there exists some $n \in \N$ such that $\bfM(0)^n$
        has positive (possibly infinite) entries only.
    \item[(A2)]
        Either $\bfM(0)$ has an infinite entry or
        $\bfM(0)$ has finite entries only and Perron-Frobenius eigenvalue $\rho > 1$.
    \item[(A3)]
        There exists some $\alpha > 0$ such that $\bfM(\alpha)$ has finite entries only
        and 1 is the Perron-Frobenius eigenvalue of $\bfM(\alpha)$
        with left and right eigenvectors $\bu = (u_1,\ldots,u_p)$ and $\bv = (v_1,\ldots,v_p)$.
    \item[(A4)]
        $-\im_{j}'(\alpha) := \int_{[0,\infty)} t e^{-\alpha t} \imu(\{j\}\times\dt) \in (0,\infty)$ for $i,j=1,\ldots,p$.
\end{itemize}
The following assumption will only be in force when explicitly stated:
\begin{itemize}
    \item[(A5)]
            $\E \! \Big[\big(\int_{[0,\infty)} e^{-\alpha t} \iZ(\dj \times \dt)\big)  \log^+ \!\! \big(\int_{[0,\infty)}
            e^{-\alpha t} \iZ(\dj \times \dt)\big)\Big] \! < \infty$
            for $i=1,\ldots,p$.
\end{itemize}
Although the nonlattice assumption which forms a part of
(A1) may appear restrictive at first sight, it is not, for it
holds whenever one of the $\imu(\{j\} \times \cdot)$ has a
nontrivial continuous component.
With little effort, the results of the paper can be extended to
the lattice case. While (A2) entails supercriticality, (A3)
demands the existence of a Malthusian parameter. By convention, we
assume that $\bu$ and $\bv$ are such that $\bu \cdot \bv^{\transp}
= \mathbf{1} \cdot \bv^{\transp} = 1$  (\,$^{\transp}$ for
transpose) or, more explicitly,
\begin{equation}    \label{eq:uv}
\sum_{i=1}^p u_{i} v_{i}    ~=~ \sum_{i=1}^p v_{i} ~=~  1.
\end{equation}
Finally we note that (A4) is a drift condition, whereas (A5) is the classical $(Z \log Z)$-condition for the multi-type branching random walk.

For $i=1,\ldots,p$, define
\begin{equation}    \label{eq:multi-type martingale}
\iWn    ~:=~    \sum_{|x|=n} \frac{v_{\tau(x)}}{v_i} e^{-\alpha S(x)},  \quad n \in \N_0.
\end{equation}
It can be checked that $(\iWn)_{n \geq 0}$ is a nonnegative mean-one martingale under $\Prob^{i}$
with respect to the canonical filtration.
Hence, it converges $\Prob^{i}$-a.s.~to some finite random variable $\iW \geq 0$.
If (A5) holds, then $\Prob^i(\iW>0)>0$, see Section \ref{subsec:basic martingale convergence} for details.

\subsection{Convergence in probability and $\L^1$}

Our starting point is the following weak law of large numbers.

\begin{Thm} \label{Thm:WLLN}
Assume that $t \mapsto e^{-\alpha t} \E^i [\phi(t)]$ is directly Riemann integrable\footnote{
See p.~232 in \cite{Resnick:2002} for the definition of direct Riemann integrability.} and that, for each $i=1,\ldots,p$,
\begin{equation}    \label{eq:A1 WLLN}
\E^i \Big[ \sup_{0 \leq s \leq t} \phi(s) \Big] ~<~ \infty  \quad   \text{for all } t \geq 0.
\end{equation}
Then, for each $i=1,\ldots,p$,
\begin{equation}    \label{eq:WLLN}
e^{-\alpha t} \Z^{\phi}(t)
~\to~   \iW \frac{v_i \sum_{j=1}^p u_j \int_0^{\infty}  e^{-\alpha s} \E^j[\phi(s)] \, \ds}{\sum_{j,k=1}^p u_j  v_k (-\jm_k)'(\alpha)}
\quad   \text{in } \Prob^i \text{-probability as } t \to \infty.
\end{equation}
If (A5) is valid, then the above convergence also holds in $\L^1(\Prob^i)$.
\end{Thm}

This result has been derived before by Nerman \cite{Nerman:1979}, Jagers
\cite{Jagers:1989} and Olofsson \cite{Olofsson:2009} (the last two
in the more general situation of an abstract type space). We
include a new proof since the methods employed in the proof of our
main result, the rate of convergence in \eqref{eq:WLLN}, lead to a short and simple derivation of \eqref{eq:WLLN}.

\subsection{Almost sure convergence}

\begin{Con} \label{Con:SLLN mu}
For some $\varepsilon > 0$ and $g(x) := x \log^{1+\varepsilon}(1+x)$, $x \geq 0$
\begin{equation}    \label{eq:SLLN mu}
\int g(t) e^{-\alpha t} \imu(\{j\} \times \dt) ~=~  \E^i\bigg[\sum_{|x|=1: \tau(x)=j} e^{-\alpha S(x)} g(S(x)) \bigg]   ~<~ \infty
\end{equation}
for $i,j=1,\ldots,p$.
\end{Con}

\begin{Con} \label{Con:SLLN phi}
For some $\varepsilon > 0$ and $h(x) := x\log^{1+\varepsilon}(1+x)$, $x \geq 0$
\begin{equation}    \label{eq:SLLN phi}
\sup_{t \geq 0}  \, (h(t) \vee 1) e^{-\alpha t} \phi(t)
\end{equation}
has finite expectation with respect to $\Prob^i$ for $i=1,\ldots,p$.
\end{Con}

\begin{Thm} \label{Thm:SLLN}
Assume that Conditions \ref{Con:SLLN mu} and \ref{Con:SLLN phi} are satisfied and that
$\phi$ has paths in the Skorokhod space $D := D(\R)$ of right-continuous functions with finite left limits.
Then, for each $i=1,\ldots,p$,
\begin{equation}    \label{eq:WLLN}
e^{-\alpha t} \Z^{\phi}(t)
~\to~   \iW \frac{v_i \sum_{j=1}^p u_j \int_0^{\infty}  e^{-\alpha s} \E^j[\phi(s)] \, \ds}{\sum_{j,k=1}^p u_j  v_k (-\jm_k)'(\alpha)}
\quad   \Prob^i \text{-a.s.~as } t \to \infty.
\end{equation}
\end{Thm}

This result is Theorem \cite[Theorem 6.6]{Nerman:1979}.
We reprove it using a different method, and slightly stronger assumptions, than in \cite{Nerman:1979}.
However, for almost all applications, Conditions \ref{Con:SLLN mu} and \ref{Con:SLLN phi} will be sufficiently weak
and we think that our formulation of Theorem \ref{Thm:SLLN} constitutes a fair tradeoff
between trying to go as close as possible to the optimal conditions and trying to keep the arguments short.

\subsection{Ratio convergence}

Theorem \ref{Thm:Ratio convergence} below provides sufficient
conditions for the convergence of the ratio of two processes
$\Z^\phi$ and $\Z^\psi$. The theorem is interesting mainly when
the $(Z \log Z)$-condition (A5) fails. It is an extension of
Theorem 6.3 in \cite{Nerman:1981} to the multi-type case and
follows quite easily from the methods used here. A similar
result can be found in \cite[Theorem 6.7]{Nerman:1979}; however, there
convergence is shown under assumptions that are too restrictive
for the future applications we have in mind.

\begin{Con} \label{Con:Ratio convergence mu}
There is some $\theta < \alpha$ such that $\bfM(\theta)$ has finite entries only.
\end{Con}

\begin{Con} \label{Con:Ratio convergence phi}
$\phi$ is not identically $0$ with positive probability and has
paths in the Skorokhod space $D := D(\R)$ of right-continuous
functions with finite left limits and there exists a $\theta < \alpha$ such that, for $i=1,\ldots,p$,
\begin{equation*}
\E^i \Big[\sup_{t \geq 0} e^{-\theta t} \phi(t) \Big]   ~<~ \infty.
\end{equation*}
\end{Con}

\begin{Thm} \label{Thm:Ratio convergence}
Assume that Condition \ref{Con:Ratio convergence mu} holds and
that $\phi$ and $\psi$ are characteristics satisfying Condition
\ref{Con:Ratio convergence phi}. Then, on $S = \{|\G_n|>0 \text{
for all } n \in \N_0\}$, for $i=1,\ldots,p$,
\begin{equation}    \label{eq:ratio convergence}
\frac{\Z^{\phi}(t)}{\Z^{\psi}(t)}
~\to~   \frac{\sum_{j=1}^p u_j \int_0^{\infty}  e^{-\alpha s} \E^j[\phi(s)] \, \ds}{\sum_{j=1}^p u_j \int_0^{\infty}  e^{-\alpha s} \E^j[\psi(s)] \, \ds}
\quad   \Prob^i \text{-almost surely as } t \to \infty.
\end{equation}
\end{Thm}

\subsection{Rate of convergence}    \label{subsec:rates of convergence}

Let $\delta > 0$. The following conditions are needed to formulate our main result, Theorem \ref{Thm:Rates of convergence}.
\begin{Con} \label{Con:Rate of convergence Z}
\begin{equation*}
\E^i \big[\iWone (\log^+ \iWone)^{1+\delta}\big]    ~<~ \infty
\quad   \text{for } i=1,\ldots,p.
\end{equation*}
\end{Con}

\begin{Con}     \label{Con:spread-out}
Assume that there is a finite sequence $i_0,\ldots,i_n \in \{1,\ldots,p\}$ such that the convolution
\begin{equation*}
\,^{\mathit{i}_0}\!\mu(\{i_1\} \times \cdot) * \ldots * \,^{\mathit{i}_{n-1}}\!\mu(\{i_n\} \times \cdot)
\end{equation*}
possesses a nontrivial component which is absolutely continuous with respect to the Lebesgue measure.
\end{Con}

\begin{Rem} \label{Rem:spread-out}
In the single-type case ($p=1$), Condition \ref{Con:spread-out}
says that the distribution $e^{-\alpha t} \onemu(\{1\} \times
\dt)$ is spread-out.\!\footnote{A finite measure $\mu$ on $\R$ is
called \emph{spread-out} if some convolution power $\mu^{*n}$ of
$\mu$ has a nontrivial component which is absolutely continuous
with respect to Lebesgue measure.}
\end{Rem}

\begin{Con} \label{Con:Rate of convergence Z^phi}
There exists an eventually increasing function\footnote{A function
$h: \R_{\geq 0} \to \R$ is called \emph{increasing} if $s
\leq t$ implies $h(s) \leq h(t)$ for all $s,t \geq 0$. It is
called \emph{eventually increasing} if for some $a \geq 0$, $h$ is
increasing on $[a,\infty)$. $h$ is called \emph{decreasing} or
\emph{eventually decreasing} if $-h$ is increasing or eventually
increasing, respectively.} $h: \R_{\geq 0} \to (0,\infty)$
that is regularly varying of index $1$ at $\infty$ with the
properties that (i) $t \mapsto t/h(t)$ is eventually decreasing,
(ii) $t \mapsto t^2/h(t)$ is eventually increasing, and (iii) $t
(\log t)^{2\delta} = o(h(t))$ as $t \to \infty$ such that, for
$i=1,\ldots,p$,
\begin{equation}    \label{eq:Rate of convergence Z^phi}
\sup_{t \geq 0} \E^i \big[ h\big(e^{-\alpha t} \Z^{\phi}(t)\big)\big]   ~<~ \infty.
\end{equation}
\end{Con}

For a particular $\phi$ sufficient conditions for \eqref{eq:Rate
of convergence Z^phi} to hold are given in the proof of Theorem
6.1. For general $\phi$ finding such sufficient conditions is a
problem on its own which does not seem simple, and we refrain from
investigating it here.

\begin{Rem} \label{Rem:Rate of convergence Z^phi}
$h(t)=t (\log t)^{2 \delta}\log (\log t)$ (for large $t$)
is a typical example of the function $h$ in Condition
\ref{Con:Rate of convergence Z^phi}.
\end{Rem}

\begin{Con}     \label{Con:Rate of convergence phi}
For $i=1,\ldots,p$, the mapping $t \mapsto e^{-\alpha t} \E^i
[\phi(t)]$ is bounded and Lebesgue integrable with
\begin{equation}    \label{eq:Rate of convergence phi}
\lim_{t \to \infty} t^{\delta} \int_t^{\infty} e^{-\alpha s} \E^i[\phi(s)] \,\ds = 0
\quad   \text{and}  \quad
\lim_{t \to \infty} t^{\delta} \sup_{s \geq t} e^{-\alpha s} \E^i [\phi(s)] = 0.
\end{equation}
\end{Con}

\begin{Thm} \label{Thm:Rates of convergence}
Assume that, for some $\delta > 0$, Conditions \ref{Con:Rate of
convergence Z}, \ref{Con:spread-out}, \ref{Con:Rate of convergence
Z^phi} and \ref{Con:Rate of convergence phi} are valid and that,
for each $i=1,\ldots,p$,
\begin{equation}    \label{eq:ES_1^(1+delta)<infty}
\E^i \bigg[\sum_{|x|=1} e^{-\alpha S(x)} S(x)^{1+\delta}\bigg] ~<~ \infty.
\end{equation}
Then, for each $i=1,\ldots,p$, in $\Prob^i$-probability,
\begin{equation}    \label{eq:Rate in WLLN}
\lim_{t \to \infty} t^{\delta} \bigg|e^{-\alpha t} \Z^{\phi}(t)
- \iW \frac{v_i \sum_{j=1}^p u_j \int_0^{\infty}  e^{-\alpha s} \E^j[\phi(s)] \, \ds}{\sum_{j,k=1}^p u_j  v_k (-\jm_k)'(\alpha)}\bigg|
~=~ 0.
\end{equation}
\end{Thm}

The rest of the paper is organized as follows. The proofs of the
main results are given in Section \ref{sec:proofs}. They are based
on an embedding technique that is set out in Section
\ref{sec:embedded process}. In Section \ref{sec:martingale
convergence}, we derive auxiliary results concerning two
martingales that are important in our analysis, namely, the
additive martingale in the multi-type branching random walk and
Nerman's martingale in continuous time. For the former, we prove
$\log$-type moment results, for the latter, we derive the exact
polynomial rate of convergence.

\section{The embedded single-type process}      \label{sec:embedded process}

The basic idea in this paper is to derive the results in the
multi-type case from the corresponding single-type ones by
considering the embedded process of type-$i$ individuals.
In this section we prove some auxiliary results which are needed
to construct the latter process.

\subsection{Change of measure}	\label{subsec:measure change}

For $i=1,\ldots,p$, we define the finite-dimensional distributions
of a sequence $((M_n,S_n))_{n \geq 0}$ under $\Prob^i$ on
$\{1,\ldots,p\} \times \R_{\geq 0}$ via the identity
\begin{align}
\E^i& [h ((M_0,S_0),\ldots,(M_n,S_n))]  \notag  \\
&=~ \E^i \bigg[\sum_{|x|=n} e^{-\alpha S(x)} \frac{v_{\tau(x)}}{v_i} h((i,0),(\tau(x|_1),S(x|_1)),\ldots,(\tau(x),S(x)))\bigg], \label{eq:M_n,S_n}
\end{align}
where $h: (\{1,\ldots,p\} \times \R_{\geq0})^{n+1} \to [0,\infty)$ is (Borel-) measurable.
The right-hand side of \eqref{eq:M_n,S_n} equals $1$ for $h \equiv 1$
because the sequence $(\iWn)_{n \geq 0}$ defined in \eqref{eq:multi-type martingale} is a mean-one martingale.
$\bfM(\alpha) \bv^{\transp} = \bv^{\transp}$ guarantees that \eqref{eq:M_n,S_n} defines a consistent family of finite-dimensional distributions.
One can further check using induction on $n$ that $((M_n,
S_n))_{n \geq 0}$ is a Markov random
walk\footnote{$((M_n,S_n))_{n \geq 0}$ is a \emph{Markov random
walk} or \emph{Markov additive process} on $\{1,\ldots,p\} \times
\R$ if $((M_n,S_{n+1}-S_n))_{n \geq 0}$ is a time-homogeneous
Markov chain on $\{1,\ldots,p\} \times \R$ for which the
transition probabilities depend only on the first coordinate,
\textit{cf.}~\cite{Pyke:1961}.} with initial distribution
$\Prob^i((M_0,S_0) = (i,0))=1$ and transition kernel
\begin{align}
\Prob^i&((M_{n+1},S_{n+1}-S_n) \in \{k\} \times B | M_n=j)  \notag  \\
&=~ \frac{v_k}{v_j} \int_B e^{-\alpha t} \jmu(\{k\} \times \dt)
~=~ \frac{v_k}{v_j} \E^j \!\bigg[\sum_{|x|=1} e^{-\alpha S(x)}
\1_{\{\tau(x)=k, S(x) \in B\}} \! \bigg]    \label{eq:MRW
transition}
\end{align}
for $j,k \in \{1,\ldots,p\}$ and $B \subseteq \R_{\geq0}$ Borel.
For later use, we list a few properties of $(M_n)_{n \geq 0}$.

\begin{Lemma}   \label{Lem:M}
Fix $i \in \{1,\ldots,p\}$ and let $\sigma^i := \inf\{n > 0: M_n = i\}$.
\begin{itemize}
    \item[(a)]
        Under $\Prob^i$, $(M_n)_{n \geq 0}$ is a Markov chain with probability of transition from $j$ to $k$ given by
        $\jm_k(\alpha) v_k/v_j$ and stationary distribution $\pi = (\pi_1,\ldots,\pi_p)$
        where $\pi_j = u_j v_j$, $j=1,\ldots,p$.
    \item[(b)]
        $\E^i[\#\{0 \leq k < \sigma^i: M_k = j\}] = \E^i[\sigma^i] u_j v_j$ and $\E^i[\sigma^i] = (u_i v_i)^{-1}$.
    \item[(c)]
        For some $\gamma > 0$, $\E^i [e^{\gamma \sigma^i}] < \infty$ for $i=1,\ldots,p$.
\end{itemize}
\end{Lemma}
\begin{proof}
The first statement in (a) follows from \eqref{eq:MRW transition},
the second from
\begin{equation*}
\sum_{j=1}^p \pi_j \frac{\jm_k(\alpha) v_k}{v_j}      ~=~
\sum_{j=1}^p u_j \jm_k(\alpha) v_k ~=~ (\bu \bfM(\alpha))_k v_k
~=~ u_k v_k ~=~ \pi_k
\end{equation*}
and the fact that $\pi$ is normed by convention, see
\eqref{eq:uv}. For the proof of (b), define $\tilde{\pi}_j :=
\E^i[\#\{0 \leq k < \sigma^{i}: M_k = j\}]$, $j=1,\ldots,p$,
and $\tilde{\pi} := (\tilde{\pi}_j)_{j=1,\ldots,p}$. It is known
that $\tilde{\pi}$ is a left eigenvector to the eigenvalue $1$ for
the transition matrix $(\jm_k(\alpha) v_k/v_j)_{j,k=1,\ldots,p}$.
Hence, $\tilde{\pi}=c\pi$ for some $c>0$. Further, $\sum_{j=1}^p
\tilde{\pi}_j = \E^i[\sigma^i]$, $\sum_{j=1}^p \pi_j = 1$
and $\tilde{\pi}_i=1$ imply $c=\E^i[\sigma^i] = \pi_i^{-1}$
and hence assertion (b). Finally, by (A1), there exists an $n \in
\N$ such that $\bfM(\alpha)^n$ has positive entries only. Let $d$
be the minimal entry of the matrix $\bfM(\alpha)^n$. Then
$\Prob^i(\sigma^i > kn) \leq (1-d)^k$ for all $k \in
\N_0$. From this, assertion (c) is easily deduced.
\end{proof}

\subsection{Optional lines} \label{subsec:optional lines}

We make use of particular \emph{optional lines} (see
\cite{Biggins+Kyprianou:2004, Jagers:1989} for a general
treatment).

Fix $i \in \{1,\ldots,p\}$ and let $\sigma^i$ be defined as in
Lemma \ref{Lem:M}, \textit{i.e.}, $\sigma^i := \inf\{k > 0: M_k = i\}$.
Associated with $\sigma^i$ is an optional line
$\J^i \subseteq \G$ defined by
\begin{equation*}
\J^i ~:=~ \{x \in \G\setminus \{\varnothing\}: \tau(x) = i,\, \tau(x|_j) \not = i \text{ for } 0 < j < |x| \}.
\end{equation*}
Further, let $\sigma^i_n$ be the $n$th consecutive application of $\sigma^i$,
\textit{i.e.}, $\sigma^i_0 := 0$ and $\sigma^i_n := \inf\{k > \sigma^i_{n-1}: M_k = i\}$.
The optional lines associated with the $\sigma^i_n$ are denoted by $\J^i_n$, \textit{i.e.},
$\J^i_0 := \{\varnothing\}$ and
\begin{equation*}
\J^i_n  ~:=~    \bigcup_{x \in \J^i_{n-1}} \{xy: y \in [\J^i]_x\}.
\end{equation*}
Notice that the $\J^i_n$ as defined here are optional lines in the
sense of \cite{Jagers:1989} and \emph{very simple lines} in the
sense of \cite[Section 6]{Biggins+Kyprianou:2004}. Jagers
\cite[Theorem 4.14]{Jagers:1989} established the strong Markov
property for branching processes along optional lines, a result
that is crucial for our arguments here.

One can check using \eqref{eq:M_n,S_n} that
\begin{equation}    \label{eq:sigma=n<->J}
\Prob^i[((M_0,S_0),\ldots,(M_n,S_n)) \! \in \! B, \sigma^i = n] =
\E^i \! \bigg[\sum_{|x|=n\,:\, x \in \J^{i}} e^{-\alpha
S(x)} \frac{v_{\tau(x)}}{v_i} \delta_{\bT \otimes \bS(x)}(B)\bigg]
\end{equation}
and
\begin{equation}    \label{eq:sigma>n<->J}
\Prob^i[((M_0,S_0),\ldots,(M_n,S_n)) \! \in \! B, \sigma^i > n] =
\E^i \! \bigg[\sum_{|x|=n\,:\, x \prec \J^{i}} e^{-\alpha
S(x)} \frac{v_{\tau(x)}}{v_i}\delta_{\bT \otimes \bS(x)}(B)\bigg]
\end{equation}
for $B \subseteq (\{1,\ldots,p\} \times \R_{\geq 0})^{n+1}$
Borel and $\bT \otimes \bS(x) := ((\tau(x|_k),S(x|_k)))_{0 \leq k
\leq |x|}$. Summation over $n \geq 0$ and a standard approximation
argument give
\begin{equation}    \label{eq:sigma<->J general}
\E^i [f(M_{\sigma^i},S_{\sigma^i})] =   \E^i \bigg[\sum_{x \in
\J^{i}} e^{-\alpha S(x)} \frac{v_{\tau(x)}}{v_i}
f(\tau(x),S(x))\bigg]
\end{equation}
and
\begin{equation}    \label{eq:sigma<->prec J general}
\E^i \bigg[\sum_{k=0}^{\sigma^i-1} f(M_{k},S_{k})\bigg]
= \E^i \bigg[\sum_{x \prec \J^{i}} e^{-\alpha S(x)}
\frac{v_{\tau(x)}}{v_i} f(\tau(x),S(x))\bigg]
\end{equation}
for every measurable function $f:\{1,\ldots,p\}\times \R_{\geq 0} \to \R_{\geq 0}$.

For ease of notation, in the subsequent proofs, we shall fix $i=1$
(the type of the ancestor). This constitutes no loss of
generality. We shall write $\Prob$ for $\Prob^1$, $\E$ for $\E^1$,
$\sigma_n$ for $\sigma^1_n$, $\J$ for $\J^1$, $\J_n$ for $\J^1_n$,
\textit{etc.}

By $\Z_{\J_n}$, we denote the point process $\sum_{x \in \J_n} \delta_{S(x)}$,
by $\mu_n$ its associated intensity measure, and by $m_n$ the Laplace transform of $\mu_n$.
We write $\Z_{\J}$, $\mu$ and $m$ when $\J = \J_1$.
Further, for $n \in \N_0$, we define
\begin{equation}    \label{eq:W_n}
V_n ~:=~    \int_{[0,\infty)} e^{-\alpha t} \Z_{\J_n}(\dt)
~=~ \sum_{x \in \J_n} e^{-\alpha S(x)}.
\end{equation}
$(V_n)_{n \geq 0}$ is a nonnegative martingale w.r.t.~the
canonical filtration and converges a.s.~to a limit variable $V
\geq 0$. In the following proposition, we establish that
$(\Z_{\J_n})_{n \in \N_0}$ fulfills the standing assumptions given
in p.~366 of \cite{Nerman:1981} which correspond to the
assumptions (A1)--(A4) in the case $p=1$ here.

\begin{Prop}    \label{Prop:J inherits1}
Assume that (A1)--(A4) hold. Then:
\begin{itemize}
    \item[(a)]  $\mu$ is not concentrated on any lattice $h\integers$ for $h > 0$.
    \item[(b)]  $m(0)>1$ and $m(\alpha)=1$.
    \item[(c)]  $-m'(\alpha) := \int_{[0,\infty)} u e^{-\alpha u} \mu(\du) =  (u_1v_1)^{-1} \sum_{i,j=1}^p u_i  v_j (-\im_j)'(\alpha) < \infty$.
\end{itemize}
\end{Prop}
\begin{proof}
By \eqref{eq:sigma<->J general}, (a) is equivalent to
$\Prob(S_{\sigma} \in h\integers) < 1$ for all $h > 0$. On the
other hand, $\Prob(S_{\sigma} \in h\integers) = 1$ for $h>0$
is equivalent to the existence of $h_1,\ldots,h_p \in [0,h)$
with $\Prob^i(S_1 \in h_{M_1}-h_i + h\integers)=1$ for all
$i=1,\ldots,p$, see \cite{Shurenkov:1984}. The latter is excluded by (A1).

Regarding (b), observe that by \eqref{eq:sigma=n<->J} and the recurrence of the Markov chain $(M_n)_{n \geq 0}$
\begin{equation*}
m(\alpha)   ~=~ \E \bigg[\sum_{x \in \J} e^{-\alpha S(x)}\bigg] ~=~ \Prob(\sigma < \infty)  ~=~ 1.
\end{equation*}
Further, the function $\theta \mapsto m(\theta)$ is strictly decreasing in $\theta$,
hence $m(0) > 1$.

Regarding the proof of (c),
first notice that
\begin{equation*}
-m'(\alpha) ~=~ \E \bigg[ \sum_{x \in \J} e^{-\alpha S(x)} S(x)
\bigg]  ~=~ \E [S_{\sigma}]
\end{equation*}
having utilized \eqref{eq:sigma<->J general} for the second equality.
The latter can be rewritten using standard Markov renewal theory:
\begin{equation*}
\E [S_{\sigma}] ~=~ \E[\sigma] \sum_{i=1}^p \pi_i \E^i[S_1] ~=~ \pi_1^{-1} \sum_{i=1}^p \pi_i \E^i[S_1]
\end{equation*}
where $\pi_i=u_iv_i$, $i=1,\ldots,p$ (see Lemma \ref{Lem:M}).
Using \eqref{eq:M_n,S_n}, $\E^i[S_1]$ can be written as
\begin{equation*}
\E^i[S_1]   ~=~ \sum_{j=1}^p \E^i[S_1 \1_{\{M_1=j\}}]   ~=~ \frac{1}{v_i} \sum_{j=1}^p v_j (-\im_j)'(\alpha)
\end{equation*}
which yields
\begin{equation*}
-m'(\alpha) ~=~ (u_1v_1)^{-1} \sum_{i=1}^p \sum_{j=1}^p u_i v_j (-\im_j)'(\alpha).
\end{equation*}
\end{proof}

\section{Martingale convergence}        \label{sec:martingale convergence}

For the proofs of our main results we need certain results on the martingales
$(V_n)_{n \geq 1}$, $(\iWn)_{n \geq 1}$ and Nerman's martingale, and the relations between them.

\subsection{Basic martingale convergence results}   \label{subsec:basic martingale convergence}

In this section, for the reader's convenience, we review the basic
convergence theorems for the martingales $(V_n)_{n \geq 1}$ and
$(\iWn)_{n \geq 1}$.
Let $S = \{|\G_n|>0 \text{ for all } n \in \N\}$ denote the survival set of the multi-type branching
random walk.

\begin{Prop}    \label{Prop:multi-type martingale convergence}
Fix $i \in \{1,\ldots,p\}$.
Then the following assertions are equivalent:
\begin{align*}
(a) &   \
(\iWn)_{n \geq 0} \text{ is uniformly integrable w.r.t.}~\Prob^{i}  \quad
&   (d) &   \
\Prob^i(\iW > 0) > 0.   \qquad  \\
(b) &   \
\iWn \to \iW \text{ in } \mathcal{L}^1(\Prob^{i}) \text{ as } n \to \infty. \quad
&   (e) &   \ \{\iW > 0\} = S \ \Prob^{i} \text{-a.s.}  \qquad  \\
(c) &   \
\E^{i}[\iW] = 1.    \quad
&   (f) &   \
\text{(A5) holds.}  \qquad
\end{align*}
\end{Prop}
\begin{proof}[Source]
The equivalence between (b) and (f) follows from Theorem 1 in
\cite{Kyprianou+Sani:2001}. Note that in the cited reference
Kyprianou and Sani assume Condition \ref{Con:Ratio convergence mu}
to hold, that is, that $\bfM(\beta)$ has finite entries only for
some $\beta < \alpha$. However, their proof also works when this
assumption is replaced by the present (weaker) assumption
(A4).\!\footnote{Condition \ref{Con:Ratio convergence mu} is
assumed in \cite{Kyprianou+Sani:2001} in order to conclude that
the spinal walk, which corresponds to $(S_n)_{n \geq 0}$ here, has
finite-mean increments. The latter property follows from
the proof of our Proposition \ref{Prop:J inherits1}(c).
Further, note that the drift condition in Theorem 1 of
\cite{Kyprianou+Sani:2001}, $\log \rho(\theta)-\theta
\rho'(\theta)/\rho(\theta) > 0$ (retaining their notation),
is used only to show that the drift of $(S_n)_{n \geq 0}$,
$\E[S_1]$ is positive. The latter is clear here since
$S_1>0$ a.s.} The remaining equivalences follow from standard
arguments.
\end{proof}

When $p=1$, Proposition \ref{Prop:multi-type martingale convergence} is known as Biggins' martingale convergence theorem.
Versions of this theorem have been derived by Biggins \cite{Biggins:1977},
Lyons \cite{Lyons:1997} and Alsmeyer and Iksanov \cite{Alsmeyer+Iksanov:2009} (in increasing generality).

\begin{Prop}    \label{Prop:martingale moments}
Let $b: (0,\infty) \to (0,\infty)$ be a measurable, locally bounded function that is regularly varying at $+\infty$ of positive index.
Then for
\begin{equation*}
\E^i [\iW b(\log^+ \iW)] < \infty   \quad   \text{for } i=1,\ldots,p
\end{equation*}
to hold it is sufficient that
\begin{equation*}
\E^i \big[\iWone b(\log^+ \iWone) \log^+ \iWone \big]   ~<~ \infty
\quad   \text{for } i=1,\ldots,p.
\end{equation*}
\end{Prop}

Proposition \ref{Prop:martingale moments} is the multi-type analogue of one implication of Theorem 1.4 in \cite{Alsmeyer+Iksanov:2009}
and the proof given below follows closely the proof given in \cite{Alsmeyer+Iksanov:2009}.
It is likely that the converse implication of the proposition also holds true as it is the case for $p=1$.
However, we refrained from investigating it since we only need the converse implication in the single-type case.

\begin{proof}
We shall not make use of the fact that the point processes
$\iZ(\{j\} \times \cdot )$ are concentrated on $\R_{\geq 0}$
thereby proving the proposition in a greater generality than it is stated.

The following recursive construction of the {\it modified
multi-type branching random walk} with a distinguished ray
$(\xi_n)_{n \in \N_0}$, called {\it spine}, is based on the
presentations in \cite{Alsmeyer+Iksanov:2009,Kyprianou+Sani:2001}
and, therefore, kept short here. Start with $\xi_0 :=
\varnothing$ and suppose that the first $n$ generations have been
constructed with $\xi_k$ being the spinal individual in the $k$th
generation, $k\leq n$. Now, while $\xi_n$ has children the
displacements of which relative to $\xi_n$ are given by a point
process whose law has Radon-Nikodym derivative $\sum_{x\in
\mathcal{N}(\xi_n)}{v_{\tau(x)}\over v_{\tau(\xi_n)}}e^{-\alpha
(S(x)-S(\xi_n))}$ with respect to the law of
$^{\tau(\xi_n)}\mathcal{Z}$, where $\mathcal{N}(x) := \{xy: y \in
[\G_1]_x\}$ denotes the set of children of $x$, all other
individuals of the $n$th generation produce and spread offspring
according to independent copies of $^i\mathcal{Z}$, $i=1,\ldots,p$
(i.e., in the same way as in the original multi-type BRW). All
children of the individuals of the $n$th generation form the
$(n+1)$st generation, and among the children of $\xi_n$ the next
spinal individual $\xi_{n+1}$ is picked with probability
proportional to $v_k e^{-\alpha s}$ if $s$ is the displacement of
$\xi_{n+1}$ relative to $\xi_n$ and $k$ is the type of
$\xi_{n+1}$.

Let $\widehat{\Z}_n$ denote the point process describing the positions of all members of the
$n$th generation as well as their types. We call
$(\widehat{\Z}_n)_{n\geq 0}$ {\it modified multi-type
branching random walk} associated with the original multi-type branching random walk
$(\Z_n)_{n\geq 0}$. Both, $(\Z_n)_{n\geq 0}$ and $(\widehat{\Z}_n)_{n\geq 0}$,
may be viewed as a random weighted tree with an
additional distinguished ray (the spine) for $(\widehat{\Z}_n)_{n\geq 0}$. On an
appropriate measurable space $(\Bbb{X},\B)$ specified
below, they can be realized as the same random element under two
different probability measures $\Prob^{i}$ and $\widehat{\Prob}^{i}$,
respectively.

Let $\mathcal{R} = \{(0,\xi_1,\xi_2,\ldots): \xi_k \in \N \text{ for all } k \in \N\}$ denote the set of infinite rays (starting at $0$)
and, for a subtree $t \subset \I$, let $\mathfrak{F}(t)$ be the set of functions $s:\I\to\R\cup\{\infty\}$ assigning position $s(x)\in\R$ to $x\in t$ (with $s(\varnothing)=0$) and $s(x)=\infty$ to $x\not\in t$. Further, let $\Sigma(t)$ denote the set of functions $q:\I\to \{0,1,\ldots,p\}$ assigning
type $q(x)\in\{1,\ldots,p\}$ to $x\in t$ and $q(x)=0$ to $x\not\in t$.
Then let
\begin{equation*}
\mathbb{X} := \{(t,s,q, \xi):t\subset\mathcal{I}, s \in \mathfrak{F}(t), q \in \Sigma(t), \xi \in \mathcal{R}\}
\end{equation*}
be the space of weighted rooted subtrees of $\I$ with a distinguished ray (spine).
Endow this space with the $\sigma$-field $\B:=\sigma(\B_n: n=0,1,\ldots)$, where
$\B_{n}$ is the $\sigma$-field generated by the sets
\begin{equation*}
\{(t',s', q^\prime, \xi') \in\mathbb{X}: t_n'=t_n, s'_{|t_{n}} \in B, q'_{|t_{n}}=q_{|t_{n}}
\text{ and } \xi'_{|n} = \xi_{|n}\}
\end{equation*}
where $t_n' = \{x \in t': |x| \leq n\}$, $t_n$ ranges over the subtrees $\subseteq \I$ with $\max\{|x|: x\in t_n\} \leq n$,
$q$ ranges over $\Sigma(t)$, $B$ over the Borel sets $\subseteq \R^{t_n}$ and $\xi$ over $\mathcal{R}$.
The subscript $_{|t_n}$ means restriction to the coordinates in $t_n$ while the subscript $_{|n}$ means restriction to all coordinates up to the $n$th.
Similarly, let $\F_{n}\subset \B_{n}$ denote the $\sigma$-field generated by the sets
\begin{equation*}
\{(t',s', q^\prime, \xi') \in\mathbb{X}: t_n'=t_n, s'_{|t_{n}} \in B, q'_{|t_{n}}=q_{|t_{n}}\}
\end{equation*}
where again $t_n$ ranges over the subtrees $\subseteq \I$ with
$\max\{|x|: x\in t_n\} \leq n$, $q$ ranges over $\Sigma(t)$ and
$B$ over the Borel sets $\subseteq \R^{t_n}$. Then under
$\widehat\Prob^{i}$ the identity map
$(\G,\mathbf{S},\mathbf{\tau},\xi) =
(\G,(S(x))_{x\in\I},(\tau(x))_{x\in\I}, (\xi_n)_{n\ge 0})$
represents the modified multi-type branching random walk with its
spine, while $(\G,\mathbf{S},\mathbf{\tau})$ under $\Prob^{i}$
represents the original branching random walk (the way how
$\Prob^{i}$ picks a spine does not matter and thus remains
unspecified).\!\footnote{There is a slight abuse of notation in
interpreting $\Prob^{i}$ as a distribution on $(\Bbb{X},\B)$
rather than on the product space $(\Omega,\A)$. However, we think
that introducing a new notation for this proof would be
distracting rather than clarifying.} Finally, the random variable
$^i W_n: \mathbb{X}\to \R_{\geq 0}$, defined as
\begin{equation*}
^iW_n(t,s, q, \xi) ~:=~ \sum_{x \in t\,:\, |x|=n} \frac{v_{q(x)}}{v_i} e^{-\alpha s(x)}
\end{equation*}
is $\F_{n}$-measurable for each $n \geq 0$ and satisfies
$\iWn=\sum_{|x|=n}\frac{v_{\tau(x)}}{ v_i}e^{-\alpha S(x)}$.
$\iWn$ is the Radon-Nikodym derivative of $\widehat\Prob^i$
w.r.t.~$\Prob^i$ on $\F_n$, see formula (4) in
\cite{Kyprianou+Sani:2001}. Standard theory (\textit{cf.}~Lemmas
5.1 and 5.2 in \cite{Alsmeyer+Iksanov:2009}) yields that the
martingale $(\iWn)_{n\in\N_0}$ is uniformly $\Prob^{i}$-integrable
if and only if $\widehat\Prob^{i}\{\iW<\infty\}=1$ and, in the
case of uniform integrability,
\begin{equation}    \label{8}
\E^i [\iW h(\iW)] ~=~ \widehat\E^i [h(\iW)]
\end{equation}
for each nonnegative Borel function $h$ on $[0,\infty)$.

For $x\in \G$, put $L(x)=\frac{v_{\tau(x)}e^{-\alpha S(x)}}{v_{\tau(\varnothing)}}$ and notice that, if $|x|=k$,
\begin{equation*}
[\,^{\tau(\varnothing)}W_n]_x ~=~ \sum_{y:xy\in
\mathcal{G}_{k+n}}\frac{L(xy)}{L(x)},    \quad n=0,1,\ldots
\end{equation*}
Since all individuals off the spine reproduce and spread as in the
original multi-type BRW, we have that, under $\Prob^{i}$ and
$\widehat\Prob^{i}$, the $[\iWn]_x$ for $x$ off the spine and of
type $j$ have the same distributions as $\jWn$ under
$\Prob^j$, in particular,
\begin{equation}    \label{5}
\widehat{\E}^i \big[[\iWn]_x | \tau(x)=j, \xi_{|x|} \not = x\big] ~=~ \E^i\big[ [\iWn]_x | \tau(x)=j\big] ~=~
\E^j\big[ \jWn \big]  ~=~ 1.
\end{equation}
Let $\mathcal{C}$ be the $\sigma$-field generated by the family of
types of the children of the $\xi_n$ and displacements of these
relative to their mother, \textit{i.e.}, by the family
$\{\tau(x), S(x)-S(\xi_n): x \in
\mathcal{N}(\xi_n), n=0,1,\ldots\}$. For $n \geq 1$ and
$k=1,...,n$, put
\begin{equation*}
R_{n,k} ~:=~ \sum_{x \in \mathcal{N}(\xi_{k-1}) \setminus\{\xi_k\}} \frac{L(x)}{ L(\xi_{k-1})}\big( [\,^{\tau(\varnothing)}W_{n-k}]_x-1\big).
\end{equation*}
With these definitions we can rewrite $\iWn$ as follows
\begin{eqnarray*}
\iWn
& = &
L(\xi_n) + \sum_{k=1}^{n}\sum_{x \in \mathcal{N}(\xi_{k-1})\setminus \{\xi_k\}} L(x) [\iW_{\!n-k}]_x   \\
& = &
L(\xi_n) + \sum_{k=1}^n L(\xi_{k-1})\bigg(\sum_{x\in \mathcal{N}(\xi_{k-1})}\frac{L(x)}{L(\xi_{k-1})}-\frac{L(\xi_{k})}{L(\xi_{k-1})}
+ R_{n,k}\bigg)\\
& = &
L(\xi_n) + \sum_{k=1}^n\bigg(L(\xi_{k-1})\bigg(\sum_{x \in  \mathcal{N}(\xi_{k-1})} \frac{L(x)}{L(\xi_{k-1})}+R_{n,k}\bigg)-L(\xi_k)\bigg)
\quad   \widehat\Prob^i \text{-a.s.}
\end{eqnarray*}
This implies
\begin{eqnarray}
\widehat{\E}^{i}[\iWn | \mathcal{C}] & = &
\sum_{k=0}^{n-1}\bigg(L(\xi_k)\sum_{x \in
\mathcal{N}(\xi_k)}\frac{L(x)}{ L(\xi_k)}\bigg)
-\sum_{k=1}^{n-1} L(\xi_k)      \notag  \\
& \leq &
\sum_{k=0}^{n-1}\bigg(L(\xi_k)\sum_{x \in \mathcal{N}(\xi_k)}\frac{L(x)}{ L(\xi_k)}\bigg)
\quad   \widehat\Prob^i \text{-a.s.},       \label{2}
\end{eqnarray}
for $\widehat\E^i[L(\xi_{k-1})R_{n,k}|\mathcal{C}] = L(\xi_{k-1})\widehat\E^i[R_{n,k}|\mathcal{C}]=0$
$\widehat{\Prob}_i$-a.s.~as a consequence of \eqref{5}
and since the $[\,^{\tau(\varnothing)}W_{n-k}]_x$ for $x$ off the spine are independent of $\mathcal{C}$.

According to Proposition \ref{Prop:multi-type martingale convergence},
the assumptions of the proposition ensure that the martingale $(\iWn)_{n \in \N_0}$ is
uniformly $\Prob^{i}$-integrable, hence $\widehat\Prob^{i}\{^iW<\infty\}=1$.
Passing to the limit as $n\to\infty$ in \eqref{2} and using Fatou's lemma, we get
\begin{equation}    \label{9}
\widehat \E^i[\iW|\mathcal{C}]
~\leq~ \bar{v}\sum_{k \ge 0} \bigg(e^{-\alpha S(\xi_k)}
\sum_{x\in\mathcal{N}(\xi_k)}\frac{L(x)}{L(\xi_k)}\bigg)
\quad   \widehat \Prob^{i}  \text{-a.s.}
\end{equation}
where $\bar{v}:=\max(v_1,\ldots, v_p)/\min(v_1,\ldots,v_p) \in [1,\infty)$.

Let $(A_n^{(k)}, B_n^{(k)})$, $k=1,\ldots,p, n \geq 1$
be independent under $\widehat\Prob^{i}$ with
\begin{align}   \label{4}
&\widehat\Prob^{i}\big((A_n^{(k)},B_n^{(k)}) \in B\big) \notag  \\
&=~\widehat\Prob^{i}\bigg(\!\bigg(e^{-\alpha(S(\xi_n)-S(\xi_{n-1}))},
\sum_{x\in \mathcal{N}(\xi_{n-1})} \frac{v_{\tau(x)}}{v_{\tau(\xi_{n-1})}}e^{-\alpha(S(x)-S(\xi_{n-1}))} \! \bigg)
\! \in {B}\bigg|\tau(\xi_{n-1})=k\bigg)   \notag  \\
&=~\E^k \bigg[\sum_{|x|=1}\frac{v_{\tau(x)}}{v_k}e^{-\alpha S(x)}
\1_{B} \bigg(e^{-\alpha
S(x)},\sum_{|y|=1}\frac{v_{\tau(y)}}{ v_k}e^{-\alpha S(y)}
\bigg)\bigg],
\end{align}
where $B$ is a Borel subset of $\R_{\geq 0} \times \R_{\geq 0}$. Now define
\begin{equation*}
(C_n,D_n)   ~:=~    \Big(\max_{k = 1, \ldots, p} A_n^{(k)}, \max_{k = 1, \ldots, p} B_n^{(k)}\Big),
\quad   n \geq 1.
\end{equation*}
Then the vectors $(C_n,D_n)_{n\geq 1}$ are i.i.d.
Further, for $x,y\geq 0$ and $n \geq 1$,
\begin{align*}
\widehat\Prob^{i} & \bigg(e^{-\alpha(S(\xi_n)-S(\xi_{n-1}))}\leq x,
\sum_{x\in \mathcal{N}(\xi_{n-1})}\frac{v_{\tau(x)}}{v_{\tau(\xi_{n-1})}}e^{-\alpha(S(x)-S(\xi_{n-1}))} \leq y\bigg)    \\
&=~ \sum_{k=1}^p\widehat{\Prob}^{i}(\tau(\xi_{n-1})=k) \, \widehat{\Prob}^{i}(A_n^{(k)}\leq x, B_n^{(k)}\leq y)     \\
&\geq~ \widehat{\Prob}^{i}\Big(\max_{k = 1, \ldots, p}
A_n^{(k)}\leq x, \max_{k = 1, \ldots, p} B_n^{(k)}\leq y \Big) ~=~
\widehat{\Prob}^{i}(C_n\leq x, D_n\leq y).
\end{align*}
Hence,
\begin{equation}    \label{10}
\sum_{k \geq 0} \bigg(e^{-\alpha S(\xi_k)} \sum_{x \in
\mathcal{N}(\xi_k)}\frac{L(x)}{ L(\xi_k)}\bigg)
~\stackrel{\mathrm{d}}{\leq}~     \sum_{k \geq 1}
\bigg(\prod_{j<k} C_j \bigg) D_k
\end{equation}
where ``$\stackrel{\mathrm{d}}{\leq}$'' denotes stochastic
domination (w.r.t.~$\widehat\Prob^i$ here). 

By Lemma 2.1 in \cite{Alsmeyer+Iksanov:2009}, there exist
increasing and concave functions $f$ and $g$ on $[0,\infty)$
with $f(0)=g(0)=0$ such that $b(\log x)\sim f(x)$ and $b(\log
x)\log x\sim g(x)$ as $x\to\infty$. Therefore it suffices to prove
that
\begin{equation}\label{1} \E^i[^iW_1g(^iW_1)]<\infty    \quad   \text{for } i=1,\ldots,p
\end{equation}
entails
\begin{equation*}
\E^i[^iWf(^iW)] < \infty    \quad   \text{for } i=1,\ldots,p.
\end{equation*}
Using \eqref{4} we have $\widehat\E^i [g(B_1^{(k)})] = \E^k[^kW_1g(^kW_1)] < \infty$
for $k=1,\ldots,p$, where the finiteness is secured by \eqref{1}. Hence
\begin{eqnarray}    \label{6}
\widehat\E^i [g(D_1)]
& = &
\widehat \E^i \big[g(\max_{k = 1, \ldots, p} B_1^{(k)})\big]    \notag  \\
& \leq &
\widehat \E^i \big[g(B_1^{(1)}+\ldots+ B_1^{(p)})\big]
~\leq~  \sum_{k=1}^p \widehat\E^i \big[g(B_1^{(k)})\big]
~<~ \infty,     \label{6}
\end{eqnarray}
the penultimate inequality following by subadditivity.
Since
\begin{eqnarray*}
\underline{v} e^{-\alpha(S(\xi_n)-S(\xi_{n-1}))}
& \leq &
\frac{v_{\tau(\xi_n)}}{v_{\tau(\xi_{n-1})}} e^{-\alpha(S(\xi_n)-S(\xi_{n-1}))}  \\
& \leq &
\sum_{x\in \mathcal{N}(\xi_{n-1})}\frac{v_{\tau(x)}}{v_{\tau(\xi_{n-1})}} e^{-\alpha(S(x)-S(\xi_{n-1}))}
\quad   \widehat\Prob^{i}\text{-a.s.}
\end{eqnarray*}
for each $n\geq 1$, where
$\underline{v}:=\min(v_1,\ldots,v_p)/\max(v_1,\ldots,v_p) \in (0,1]$,
we infer with the help of \eqref{4} that, for $n\geq 1$, under $\widehat\Prob^{i}$,
\begin{equation*}
\underline{v}C_n ~\stackrel{\mathrm{d}}{\leq}~  D_n.
\end{equation*}
The latter inequality in combination with \eqref{6} and the concavity of $g$ implies
\begin{equation}    \label{7}
\underline{v} \widehat\E^i [g(C_1)]
~\leq~  \widehat\E^i [g(\underline{v}C_1)]
~\leq~  \widehat\E^i [g(D_1)]
~<~ \infty.
\end{equation}
By Theorem 1.2 in \cite{Alsmeyer+Iksanov:2009}
\eqref{6} and \eqref{7} are sufficient for
\begin{equation*}
\widehat \E^i \bigg[ f\bigg(\sum_{k \geq 1} \bigg(\prod_{j<k} C_j \bigg) D_k\bigg)\bigg] ~<~    \infty
\end{equation*}
to hold. This together with \eqref{9}
and \eqref{10} yields
\begin{eqnarray}    \label{link5.4}
\E^i [^iWf(^iW)]
& = &
\widehat\E^i [f(^iW)]
~\leq~  \widehat\E^i \bigg[f\bigg(\bar{v}\bigg(\sum_{k \geq 1} \bigg(\prod_{j<k} C_j \bigg) D_k\bigg)\bigg) \bigg]  \notag  \\
& \leq & \bar{v} \widehat\E^i \bigg[f\bigg(\sum_{k \geq 1} \bigg(\prod_{j<k} C_j \bigg) D_k\bigg)\bigg]
~<~ \infty,
\end{eqnarray}
where the equality is a consequence of \eqref{8}, the first
inequality is justified by an application of Jensen's inequality
for conditional expectations, while the second follows from the
inequality $f(bx)\leq bf(x)$ which holds for fixed $b \geq 1$ and any $x>0$.
\end{proof}

\subsection{Rate of convergence of Nerman's martingale}

In this section, we assume that $p=1$, \textit{i.e.}, we are in
the single-type case. Then the martingales $(V_n)_{n \geq 0}$ and
$(\oneW_n)_{n \geq 0}$ are identical. There is a continuous-time
analogue of the martingale $(V_n)_{n \geq 0}$ which is important
in the study of the asymptotic behavior of the general branching
process. Let
\begin{equation}    \label{eq:J(t)}
\J(t) := \{x \in \G: S(x) > t,\,S(x|_k) \leq t \text{ for all } k < |x|\}
\end{equation}
and define
\begin{equation}    \label{eq:V(t)}
V(t)    ~:=~    \sum_{x \in \J(t)} e^{-\alpha S(x)},    \quad   t \geq 0.
\end{equation}
The family $(V(t))_{t \geq 0}$ can be viewed as Nerman's
martingale evaluated at certain random times. We now make this
connection precise. Order the individuals according to their times
of birth: $x_1$ is the ancestor, $x_2$ its first-born child
\textit{etc.}
In case that several births take place at the same time,
we order individuals first by generation and within generations according to
the lexicographic order.
We let $t_n:=S(x_n)$ be the time of birth of
the $n$th individual in the process. For $k \in \N$, let $Y_k :=
[V_1]_{x_k}-1$ and $\H_k:=\sigma(\Z(x_1),\ldots,\Z(x_k))$.
Further, define
\begin{equation*}
R_n := 1+\sum_{k=1}^n e^{-\alpha t_k} Y_k, \quad n \in \N.
\end{equation*}
Then $V(t) = R_{T_t}$, $t\geq 0$ where $T_t = \#\{x \in \G: S(x)
\leq t\}$ is the total number of births up to and including time
$t$. It is known (see Lemma 2.3 and Proposition 2.4 in
\cite{Nerman:1981} or Theorem 4.1 in p.~371 in
\cite{Asmussen+Hering:1983}) that $(R_n, \H_n)_{n\in\N}$ and
$(V(t), \H_{T_t})_{t \geq 0}$ are nonnegative martingales.
Furthermore, $V(t)$ and $R_n$ converge a.s., as $t \to \infty$
and $n \to \infty$, respectively, to the random variable $V$, the
a.s.~limit of Biggins' martingale $(V_n)_{n \geq 0}$, see
\textit{e.g.}~\cite[Theorem 3.3]{Gatzouras:2000}.

For the proof of Theorem \ref{Thm:Rates of convergence}, we need information about the rate of convergence of $V(t)$ to $V$.
While various results for the rate of convergence of Biggins' martingale to its limit have been established
\cite{Alsmeyer+al:2009,Iksanov:2006,Iksanov+Meiners:2010a},
we are not aware of a corresponding result for Nerman's martingale.
The following proposition provides such a result.

\begin{Prop}    \label{Prop:rate of convergence of Nerman's martingale}
Suppose that $\E [V_1 \log^+ V_1] < \infty$ and that, for some $\varepsilon > 0$,
\begin{equation}    \label{eq:T_te^alphat sufficient}
\E\bigg[\sum_{|x|=1} e^{-\alpha S(x)} S(x)
(\log^+(S(x)))^{1+\varepsilon} \bigg] < \infty.
\end{equation}
Let $\delta > 0$.
Then
\begin{equation}    \label{eq:criterion}
\lim_{t \to \infty} (\log t)^\delta \, \E[V_1(\log V_1-\log t) \1_{\{V_1>t\}}]  ~=~ 0
\end{equation}
is necessary and sufficient for
\begin{equation}    \label{eq:rate of convergence of V(t)}
\lim_{t \to \infty} t^{\delta} |V(t)-V|=0   \quad   \text{a.s.}
\end{equation}
to hold.
In particular, the simpler condition $\E [V_1 (\log^+ V_1)^{1+\delta}] < \infty$ is sufficient for \eqref{eq:rate of convergence of V(t)} to hold.
\end{Prop}

\begin{Rem} \label{Rem:T_te^alphat}
Assumption \eqref{eq:T_te^alphat sufficient} enables us to apply
Theorem 5.4 in \cite{Nerman:1981} which, with $\phi(t) =
\1_{[0,\infty)}(t)$, implies that $e^{-\alpha t} T_t \to dV$
a.s.~as $t \to \infty$ for some constant $d > 0$. In particular,
\eqref{eq:T_te^alphat sufficient} guarantees that
\begin{equation}    \label{eq:T_te^alphat}
T_t \asymp e^{\alpha t} \quad   \text{a.s.~on } S \text{ as } t \to \infty
\end{equation}
where $S = \{|\G_n|>0 \text{ for all } n \in \N_n\}$ is the
survival set, and that
\begin{equation}    \label{eq:T_te^alphat2}
e^{-\alpha t_n} \asymp n^{-1} \quad \text{a.s.~on } S \text{ as }
n \to \infty
\end{equation}
which follows on substituting $t=t_n$ in \eqref{eq:T_te^alphat}.
Actually, \eqref{eq:T_te^alphat sufficient} in Proposition
\ref{Prop:rate of convergence of Nerman's martingale} may be
replaced by the weaker Condition 5.1 in \cite{Nerman:1981} or any
other assumption
which ensures \eqref{eq:T_te^alphat}.
\end{Rem}

Before we prove the proposition, we recall a technical result stated as Lemma 4.2 in p.~37 in \cite{Asmussen+Hering:1983}.
\begin{Lemma}   \label{Lem:Abel}
Let $(\alpha_n)_{n \in \N}$, $(\beta_n)_{n \in \N}$ be sequences of real numbers with $0 < \beta_n \uparrow \infty$ as $n \to \infty$.
If  $\sum_{n\geq 1}\alpha_n \beta_n$ converges, then $\lim_{n \to \infty} \beta_n \sum_{k\geq n} \alpha_k=0$.
\end{Lemma}

The proof of the proposition is based on two lemmas.

\begin{Lemma}   \label{Lem:Rate of convergence R_n}
Suppose that $\E [V_1(\log^+V_1)^{\gamma}] < \infty$ for some
$\gamma \geq 1$ and let $\theta \in (0,\gamma]$. Then $\lim_{n \to
\infty} (\log n)^\theta (V-R_n) = 0$ a.s.~on $S$ is equivalent to
\begin{equation*}
\lim_{n \to \infty} (\log n)^\theta \sum_{k\geq n} e^{-\alpha t_{k+1}} \E [V_1\1_{\{V_1>k(\log k)^{-\theta}\}}] = 0
\quad \text{a.s.~on } S.
\end{equation*}
\end{Lemma}
\begin{proof}
Put $\widetilde{Y}_n:=Y_n\1_{\{|Y_n| \leq n(\log n)^{-\theta}\}}$
and $\varepsilon_n:=-\E[e^{-\alpha t_{n+1}} \widetilde{Y}_{n+1} | \H_n]$, $n \in \N$. We first show that
the condition $\E [V_1(\log^+V_1)^\theta]<\infty$ implies
a.s.~convergence of $\sum_{n\geq 2}(\log
n)^\theta(R_{n+1}-R_n+\varepsilon_n)$ on $S$. To this end, it
suffices to check that
    \vspace{-0.2cm}
\begin{itemize}
    \item[(a)] $\sum_{n\geq 2}\Prob(Y_n \neq \widetilde{Y}_n)<\infty$;
    \vspace{-0.2cm}
    \item[(b)] $\sum_{n\geq 2}(\log n)^{2\theta} \Var [e^{-\alpha t_{n+1}}\widetilde{Y}_{n+1}+\varepsilon_n|\H_n]<\infty$ a.s. on $S$.
\end{itemize}
    \vspace{-0.2cm}

\noindent Indeed, (b) implies that $\sum_{n\geq 2}(\log n)^\theta (e^{-\alpha t_{n+1}}\widetilde{Y}_{n+1}+\varepsilon_n)$
converges a.s.~on $S$ because the partial sums of this series
constitute an $\mathcal{L}^2$- martingale. Invoking (a) and the
Borel-Cantelli lemma, we infer a.s.~convergence of
\begin{equation*}
\sum_{n \geq 2}(\log n)^\theta (e^{-\alpha t_{n+1}} Y_{n+1}+\varepsilon_n)
~=~ \sum_{n\geq 2}(\log n)^\theta(R_{n+1}-R_n+\varepsilon_n)
\end{equation*}
on $S$.

For $x\geq e^\theta$, $f_\theta(x):=x(\log x)^{-\theta}$ is strictly increasing and continuous
and hence possesses an inverse function which we denote by $g_\theta(x)$, $x \geq (e/\theta)^\theta$.
Since $g_\theta(x) \sim x(\log x)^\theta$ as $x \to \infty$, we have, for some appropriate $c>0$,
\begin{align*}
\sum_{n \geq 2} & \Prob(Y_n\neq \widetilde{Y}_n)
= \sum_{n \geq 2} \Prob(|Y_1|>n(\log n)^{-\theta})<\infty   \\
& \text{iff}    \quad \sum_{n\geq 2}\Prob(g_\theta(|Y_1|\vee c)>n)<\infty   \\
& \text{iff}    \quad   \E [V_1 (\log^+V_1)^\theta] < \infty.
\end{align*}
This implies (a).
According to Lemma 4.2(v) in p.~372 in
\cite{Asmussen+Hering:1983}, $\sup_{n \geq 1} ne^{-\alpha t_n} < \infty$ a.s. on $S$.\!\footnote{ \label{p:footnote} The
cited lemma says that $e^{-\alpha t_n}\asymp n^{-1}$ as
$n\to\infty$ a.s.~on $S$ under the assumption $m(\beta)<\infty$
for some $\beta<\alpha$. An inspection of the proof reveals
that the latter assumption is only needed for $\inf_{n \geq 1} ne^{-\alpha t_n}>0$ a.s. on $S$
to hold}. With this at hand, we obtain
\begin{eqnarray*}
\Var[e^{-\alpha t_{n+1}}\widetilde{Y}_{n+1}+\varepsilon_n|\H_n] & = &
e^{-2\alpha t_{n+1}} \E [Y_1^2\1_{\{|Y_1|\leq n(\log n)^{-\theta}\}}]   \\
& = &
O (n^{-2} \E [Y_1^2\1_{\{|Y_1|\leq n(\log n)^{-\theta}\}}])
\end{eqnarray*}
as $n\to\infty$ a.s.~on $S$. Hence, for an appropriate $c>0$, a.s.~on $S$,
\begin{align*}
\sum_{n\geq 2} (\log n)^{2\theta} & \Var [e^{-\alpha t_{n+1}}\widetilde{Y}_{n+1}+\varepsilon_n|\H_n] < \infty   \\
& \text{if} \quad
\sum_{n\geq 2}n^{-2}(\log n)^{2\theta}\E [Y_1^2 \1_{\{|Y_1|\leq n(\log n)^{-\theta}\}}] < \infty    \\
& \text{iff}    \quad
\E \bigg[Y_1^2 \sum_{n\geq g_\theta(|Y_1|\vee c)}n^{-2}(\log n)^{2\theta}\bigg] < \infty    \\
& \text{iff}    \quad
\E \bigg[\frac{Y_1^2 (\log g_\theta(|Y_1|\vee c))^{2\theta}}{g_\theta(|Y_1|\vee c)}\bigg] < \infty  \\
& \text{iff}    \quad
\E [|Y_1| (\log^+ |Y_1|)^\theta] < \infty   \\
& \text{iff}    \quad
\E [V_1(\log^+ V_1)^\theta] < \infty.
\end{align*}
This proves (b).

Now a.s.~convergence of the series
$\sum_{n\geq 2}(\log n)^\theta (R_{n+1}-R_n+\varepsilon_n)$ on $S$
together with Lemma \ref{Lem:Abel} for $\alpha_n=R_{n+1}-R_n+\varepsilon_n$ and $\beta_n=(\log n)^\theta$ imply
\begin{equation*}
\lim_{n \to \infty} (\log n)^\theta (V-R_n+\sum_{k\geq n}\varepsilon_k) ~=~ 0   \quad   \text{a.s.~on } S.
\end{equation*}
Thus $\lim_{n \to \infty} (\log n)^\theta (V-R_n)=0$ a.s.~on $S$ is equivalent to
\begin{equation}    \label{eq:equivalent crit}
\lim_{n \to \infty} (\log n)^\theta \sum_{k\geq n}\varepsilon_k ~=~ 0   \quad   \text{a.s.~on } S.
\end{equation}
Here, for sufficiently large $k$, we have $\{|Y_1| \leq k(\log k)^{-\theta}\} = \{Y_1 \leq k(\log k)^{-\theta}\}$.
Hence, for these $k$, using that $\E [Y_1] = 0$ we obtain
\begin{eqnarray*}
\varepsilon_{k-1} & = & -\E[e^{-\alpha t_k}\widetilde{Y}_{k}|\H_{k-1}]
~=~ - e^{-\alpha t_k} \E[Y_1 \1_{\{|Y_1| \leq k(\log k)^{-\theta}\}}]    \\
& = &
- e^{-\alpha t_k} \E[Y_1 \1_{\{Y_1 \leq k(\log k)^{-\theta}\}}]  \\
& = &   e^{-\alpha t_k} \E[Y_1 \1_{\{Y_1 > k(\log
k)^{-\theta}\}}].
\end{eqnarray*}
Further, $\E[Y_1 \1_{\{Y_1 > k(\log k)^{-\theta}\}}] \sim \E[V_1 \1_{\{V_1 > k(\log k)^{-\theta}+1\}}]$ as $k \to \infty$.
Hence, in order to deduce that \eqref{eq:equivalent crit} is equivalent to
\begin{equation*}
\lim_{n \to \infty} (\log n)^\theta\sum_{k \geq n} e^{-\alpha
t_k}\E [V_1\1_{\{V_1>k(\log k)^{-\theta}\}}] ~=~ 0 \quad
\text{a.s.~on } S
\end{equation*}
it remains to check that
\begin{equation*}
\lim_{n \to \infty} (\log n)^\theta \sum_{k\geq n} e^{-\alpha
t_k} \E[V_1 \1_{\{V_1 \in (k(\log k)^{-\theta},\,k(\log
k)^{-\theta}+1]\}}] ~=~ 0  \quad \text{a.s.~on } S.
\end{equation*}
Validity of the latter relation can be seen from $e^{-\alpha
t_k} = O(k^{-1})$ a.s.~on $S$ and the following rough
estimate
\begin{align*}
\limsup_{n \to \infty} & (\log n)^\theta \sum_{k\geq n} \frac{1}{k} \E[V_1 \1_{\{V_1 \in (k(\log k)^{-\theta},k(\log k)^{-\theta}+1]\}}]    \\
& \leq~ 2 \limsup_{n \to \infty} \sum_{k \geq n} \Prob(V_1 \in (k(\log k)^{-\theta},k(\log k)^{-\theta}+1]) \\
& \leq~ 2 \limsup_{n \to \infty} \sum_{k \geq n} \Prob(V_1 > k(\log k)^{-\theta})   \\
& \leq~ 2 \limsup_{n \to \infty} \sum_{k \geq n}
\Prob(g_{\theta}(V_1\vee c) > k) ~=~ 0
\end{align*}
for appropriate $c>0$.
\end{proof}

\begin{Lemma}   \label{Lem:Rate of convergence R_n 2}
Assume that \eqref{eq:T_te^alphat sufficient} holds and let
$\gamma > 0$. Then \eqref{eq:criterion} with $\delta$ replaced by
$\gamma$ is necessary and sufficient for
\begin{equation}    \label{eq:criterion 4}
\lim_{n \to \infty} (\log n)^\gamma \sum_{k \geq n} e^{-\alpha t_{k+1}} \E [V_1\1_{\{V_1>k(\log k)^{-\gamma}\}}] ~=~ 0
\quad \text{a.s.~on } S.
\end{equation}
\end{Lemma}
\begin{proof}
In view of \eqref{eq:T_te^alphat2},
\begin{equation}    \label{eq:criterion 5}
\lim_{n \to \infty} (\log n)^\gamma \sum_{k\geq n} \frac{1}{k} \E [V_1\1_{\{V_1>k(\log k)^{-\gamma}\}}] ~=~ 0 \quad \text{a.s.~on } S
\end{equation}
is necessary and sufficient for \eqref{eq:criterion 4} to hold.
For large enough $n$ and appropriate $c>0$
\begin{equation*}
\sum_{k \geq n} \frac{1}{k} \E [V_1\1_{\{V_1>k(\log k)^{-\gamma}\}}]
~=~ \E \bigg[V_1\sum_{k\geq n} \frac{1}{k} \1_{\{g_\gamma(V_1\vee c)>k\}}\bigg].
\end{equation*}
Since $\sum_{k=1}^n k^{-1} = \log n + O(1)$ as $n \to \infty$,
\eqref{eq:criterion 5} is equivalent to
\begin{equation*}
\lim_{n \to \infty} (\log n)^\gamma \, \E [V_1(\log g_\gamma(V_1\vee c)-\log n)\1_{\{g_\gamma(V_1\vee c)>n\}}]  ~=~ 0
\end{equation*}
and hence to
$\lim_{n \to \infty} (\log n)^\gamma \, \E [V_1(\log g_\gamma(V_1)-\log g_\gamma(n))\1_{\{V_1>n\}}] = 0$
on substituting $g_\gamma(n)$ instead of $n$ and then simplifying.
Finally, the latter relation is equivalent to
\begin{equation*}
\lim_{t \to \infty} (\log t)^\gamma \, \E [V_1(\log g_\gamma(V_1)-\log g_\gamma(t))\1_{\{V_1>t\}}] = 0
\end{equation*}
by a monotonicity argument. It remains to note that, as $t \to
\infty$, $\log g_\gamma(t) = \log t + \gamma \log (\log t) + o(1)$
and that $\log(\log u) - \log (\log v) = o(\log u-\log v)$ as $u,v \to
\infty$ to complete the proof.
\end{proof}

\begin{proof}[Proof of Proposition \ref{Prop:rate of convergence of Nerman's martingale}]
We first prove that
\begin{equation}    \label{eq:implication}
\E [V_1(\log^+V_1)^{1+\delta}] < \infty
\quad \! \Rightarrow \! \quad
\eqref{eq:criterion}
\quad \! \Rightarrow \! \quad
\E [V_1(\log^+V_1)^{1+\delta-\varepsilon}] < \infty
\end{equation}
for any $\varepsilon \in (0,1+\delta)$.
In particular, the first implication justifies the last statement of the proposition.

Suppose $\E [V_1(\log^+V_1)^{1+\delta}] < \infty$.
Then
\begin{align*}
\lim_{t \to \infty} (\log t)^\delta \E[V_1(\log V_1-\log t) \1_{\{V_1>t\}}]
\leq    \lim_{t \to \infty} \E[V_1(\log V_1)^{1+\delta} \1_{\{V_1>t\}}] =   0,
\end{align*}
that is, \eqref{eq:criterion} holds.

Suppose \eqref{eq:criterion} and let $\varepsilon \in (0,\delta)$.
Then
\begin{equation*}
t^{-1} (\log t)^{\delta-\varepsilon-1} \E [V_1(\log V_1-\log t)\1_{\{V_1>t\}}] ~\leq~ {\rm const}\,t^{-1}(\log t)^{-\varepsilon-1}
\end{equation*}
for large enough $t$ whence
\begin{eqnarray*}
\infty
& > &
(\delta-\varepsilon)\int_1^\infty t^{-1}(\log t)^{\delta-\varepsilon-1} \E [V_1(\log V_1-\log t) \1_{\{V_1>t\}}] \, \dt \\
& = &
\E \bigg[V_1\log V_1\int_1^{V_1}(\delta-\varepsilon)t^{-1}(\log t)^{\delta-\varepsilon-1}{\rm d}t\1_{\{V_1>1\}} \bigg]  \\
& &
-\E \bigg[V_1\int_1^{V_1}(\delta-\varepsilon)t^{-1}(\log t)^{\delta-\varepsilon} \dt \1_{\{V_1>1\}}\bigg]   \\
& = &
(1+\delta-\varepsilon)^{-1}\E [V_1(\log^+ V_1)^{1+\delta-\varepsilon}]
\end{eqnarray*}
which completes the proof of \eqref{eq:implication}.

$\lim_{t \to \infty} t^\delta |V(t)-V| = 0$ a.s.~on $S^{\comp}$
holds trivially. In view of \eqref{eq:T_te^alphat2},
for $\lim_{t \to \infty} t^\delta |V(t)-V|=0$ a.s.~on $S$ to hold it
is necessary and sufficient that
\begin{equation}    \label{eq:Rate of V-R_n}
\lim_{n \to \infty} (\log n)^\delta (V-R_n) = 0 \quad\text{a.s.~on } S.
\end{equation}
Therefore we work towards proving that condition \eqref{eq:criterion} is equivalent to \eqref{eq:Rate of V-R_n}.
While doing so we argue as in the proof of Theorem 4.1(ii) in p.~36 in \cite{Asmussen+Hering:1983}.

\noindent {\sc Case 1: $\delta\in (0,1]$}.
The result follows from Lemma \ref{Lem:Rate of convergence R_n} with
$\gamma=1$ and $\theta=\delta$ which is applicable because $\E [V_1\log^+V_1] < \infty$
by the assumption and Lemma \ref{Lem:Rate of convergence R_n 2} with $\gamma=\delta$.

\noindent {\sc Case 2: $\delta>1$}.
First assume that \eqref{eq:criterion} holds.
Then $\E [V_1(\log^+V_1)^\delta] < \infty$ by \eqref{eq:implication}.
The result now follows from Lemmas \ref{Lem:Rate of convergence R_n} and \ref{Lem:Rate of convergence R_n 2} with $\gamma=\theta=\delta$.

Now assume that \eqref{eq:criterion} fails.
If $\E [V_1(\log^+V_1)^\delta]<\infty$,
then the result follows from Lemmas \ref{Lem:Rate of convergence R_n} and \ref{Lem:Rate of convergence R_n 2} with $\gamma=\theta=\delta$.
Suppose $\E [V_1(\log^+V_1)^\delta] = \infty$.
We have to prove that the relation $\lim_{n \to \infty} (\log n)^\delta(V-R_n) =0$ a.s.~on $S$ fails to hold.
Pick $\gamma \in [1,\delta)$ such that $\E [V_1(\log^+V_1)^\gamma]<\infty$,
yet $\E [V_1(\log^+V_1)^{\gamma+1/2}] = \infty$.
Using \eqref{eq:implication} with $\delta$ replaced by $\gamma$ we infer that
$\lim_{t \to \infty} (\log t)^\gamma\E [V_1(\log V_1-\log t)\1_{\{V_1>t\}}] = 0$ does not hold.
According to Lemma \ref{Lem:Rate of convergence R_n} and Lemma \ref{Lem:Rate of convergence R_n 2}
the relation $\lim_{n \to \infty} (\log n)^\gamma(V-R_n) =0$ a.s.~on $S$ does not hold.
This finishes this part of the proof, for $\gamma<\delta$.
\end{proof}

\subsection{Reduction to the single-type case}  \label{subsec:reduction to single-type}

Recall that $\Prob$, $\E$, $\sigma$ and $\J$ are shorthand notation for $\Prob^1$, $\E^1$, $\sigma^1$ and $\J^1$, respectively.
\begin{Prop}    \label{Prop:J inherits2}
Assume that (A1)--(A4) hold.
Then
\begin{itemize}
	\item[(a)]	If (A5) holds, then $V = \oneW$ $\Prob$-a.s.~and $\E [V_1 \log^+ V_1] < \infty$.
	\item[(b)]	If Condition \ref{Con:Ratio convergence mu} holds, then $m(\beta)<\infty$ for some $\beta < \alpha$.
	\item[(c)]	If Condition \ref{Con:SLLN mu} holds, then
				\begin{equation*}
				\E \bigg[\sum_{x \in \J} e^{-\alpha S(x)} g(S(x)) \bigg] < \infty
				\end{equation*}
				where $g$ is the function from Condition \ref{Con:SLLN mu}.
    \item[(d)]	If, for some $\delta > 0$, \eqref{eq:ES_1^(1+delta)<infty} holds, \textit{i.e.}, if
            	    	$\E^i \Big[\sum_{|x|=1} e^{-\alpha S(x)} S(x)^{1+\delta}\Big] < \infty$ for $i=1,\ldots,p$,
           		     	then
            	    	\begin{equation*}
              	  	\E \bigg[\sum_{x \in \J} e^{-\alpha S(x)} S(x)^{1+\delta}\bigg] < \infty.
            	    	\end{equation*}
    \item[(e)]	If, for some $\delta > 0$, Condition \ref{Con:Rate of convergence Z} holds,
              	  	then $\E [V_1 (\log^+ V_1)^{1+\delta}] < \infty$.
\end{itemize}
\end{Prop}
\begin{proof}
(a) If (A5) holds, then $\E [\oneW] = 1$ by Proposition \ref{Prop:multi-type martingale convergence}.
We only need to prove that $V=\oneW$ $\Prob$-a.s.~because then $\E[V]=1$ and hence,
by Proposition \ref{Prop:multi-type martingale convergence} (for $p=1$), $\E [V_1 \log^+ V_1] < \infty$.
But $V = \oneW$ $\Prob$-a.s.~follows from \cite[Theorem 6.1]{Biggins+Kyprianou:2004}
if we can check that
\begin{equation}    \label{eq:check Biggins-Kyprianou}
\oneW_n ~=~ \lim_{k \to \infty} \sum_{x \in \J_k \wedge n} \frac{v_{\tau(x)}}{v_1} e^{-\alpha S(x)} \quad   \Prob\text{-a.s.}
\end{equation}
where $\J_k=\J_k^1$, and $\sum_{x \in \J_k \wedge n}$ means summation over the set
\begin{equation*}
\J_k \wedge n   ~:=~    \{x \in \G: \text{either } x \in \J_k
\text{ and } |x| \leq n \text{ or } x \prec \J_k \text{ and } |x|=n \}.
\end{equation*}
(In words, these are the $x$ in $\J_k$ in the first $n$ generations and the $x$ in the $n$th generation with no ancestor in $\J_k$.)
Relation \eqref{eq:check Biggins-Kyprianou} holds true
since by definition of $\J_k$, we have $|x| \geq k$ for all $x \in \J_k$ and, therefore, $\J_k \wedge n = \G_n$ for $k \geq n$.
The proof of assertion (a) is complete.

For the proof of (b), assume that Condition \ref{Con:Ratio convergence mu} holds,
that is, $\im_j(\theta)<\infty$ for all $i,j=1,\ldots,p$ and some $\theta < \alpha$.
This implies that
\begin{equation*}
c(\varepsilon) := \max_{i=1,\ldots,p} \E^i[e^{\varepsilon S_1}] < \infty
\end{equation*}
for $0 \leq \varepsilon \leq \alpha - \theta$. Further,
$c(\varepsilon) \to 1$ as $\varepsilon \downarrow 0$. By
Lemma \ref{Lem:M}(c), $\Prob(\sigma=n) \leq C^2 e^{-\gamma n}$ for
all $n \geq 0$ and some $C, \gamma > 0$.
Now pick $\beta\in[\theta,\alpha)$ such that $c(2(\alpha-\beta)) <
e^{\gamma}$. By \eqref{eq:sigma<->J general}, $m(\beta)<\infty$ is
equivalent to $\E[e^{(\alpha-\beta)S_{\sigma}}] < \infty$. For the
latter expectation, we obtain using the Cauchy-Schwarz inequality,
\begin{eqnarray*}
\E [e^{(\alpha-\beta)S_{\sigma}}]
& = &
\sum_{n \geq 0} \E \big[ \1_{\{\sigma=n\}} e^{(\alpha-\beta)S_n}\big]
~\leq~  \sum_{n \geq 0} \Prob(\sigma=n)^{1/2} \big(\E [e^{2(\alpha-\beta)S_n}]\big)^{1/2}   \\
& \leq &
C \sum_{n \geq 0} e^{-\gamma n/2} c(2(\alpha-\beta))^{n/2}  ~<~ \infty.
\end{eqnarray*}

For the proof of (c), notice that by \eqref{eq:M_n,S_n}, for $i=1,\ldots,p$,
\begin{equation*}
\E^i [g(S_1)]   ~=~ \E^i \bigg[\sum_{|x|=1} \frac{v_{\tau(x)}}{v_i} e^{-\alpha S(x)} g(S(x))\bigg]  ~<~ \infty
\end{equation*}
where the finiteness is a consequence 
of Condition \ref{Con:SLLN mu}. Lemma \ref{Lem:moments MRP} (\textit{cf.}~Remark \ref{Rem:function h})
thus yields $\E[g(S_{\sigma})] < \infty$.
Therefore, by \eqref{eq:sigma<->J general}
\begin{equation*}
\E \bigg[\sum_{x \in \J} e^{-\alpha S(x)} g(S(x)) \bigg]    ~=~  \E[g(S_{\sigma})]  ~<~ \infty.
\end{equation*}

For the proof of (d), assume that \eqref{eq:ES_1^(1+delta)<infty} holds.
Then \eqref{eq:M_n,S_n} yields
\begin{equation*}
C ~:=~ \max_{i,j=1,\ldots,p} \E^i[S_1^{1+\delta} | M_1 = j] ~<~ \infty
\end{equation*}
where the maximum is over those $i,j$ only with $\Prob^i(M_1=j)>0$.
Fix $i_1,\ldots,i_{n-1} \in \{1,\ldots,p\}$ with
$\Prob(M_1=i_1,\ldots,M_{n-1}=i_{n-1},M_n=1) > 0$ and observe that,
by Minkowski's inequality,
\begin{eqnarray*}
\E [S_n^{1+\delta} \1_{\{M_1 = i_1,\ldots,M_{n-1}=i_{n-1}, M_n = 1\}}]	
& = &
\bigg\|\sum_{k=1}^n (S_k-S_{k-1}) \1_{\{M_1 = i_1,\ldots,M_{n-1}=i_{n-1}, M_n = 1\}} \bigg\|_{1+\delta}^{1+\delta}		\\
&\leq &
\bigg(\sum_{k=1}^n \big\| (S_k-S_{k-1}) \1_{\{M_1 = i_1,\ldots,M_{n-1}=i_{n-1}, M_n = 1\}} \big\|_{1+\delta} \bigg)^{\!\!\!1+\delta}.		\\
\end{eqnarray*}
Conditioning with respect to $(M_n)_{n \geq 0}$ yields
\begin{equation*}
\E[(S_k-S_{k-1})^{1+\delta} \1_{\{M_1 = i_1,\ldots,M_{n-1}=i_{n-1}, M_n = 1\}}]
~\leq~
C \Prob(M_1 = i_1,\ldots,M_{n-1}=i_{n-1}, M_n = 1)
\end{equation*}
for $k=1,\ldots,n$.
Hence, using that $\{\sigma=n\} = \bigcup_{2 \leq i_1,\ldots,i_{n-1} \leq p} \{M_1 = i_1,\ldots,M_{n-1}=i_{n-1}, M_n = 1\}$,
\eqref{eq:sigma<->J general} and Lemma \ref{Lem:M}(c), we infer
\begin{align*}
\E & \bigg[\sum_{x \in \J} e^{-\alpha S(x)} S(x)^{1+\delta}\bigg]
~=~ \E [S_{\sigma}^{1+\delta}]  \\
&=~ \sum_{n \geq 1} \sum_{2 \leq i_1,\ldots,i_{n-1} \leq p} \E [S_n^{1+\delta} \1_{\{M_1 = i_1,\ldots,M_{n-1}=i_{n-1}, M_n = 1\}}]  \\
&\leq~ \sum_{n \geq 1} \sum_{2 \leq i_1,\ldots,i_{n-1} \leq p}
n^{1+\delta} C \Prob(M_1 = i_1,\ldots,M_{n-1}=i_{n-1}, M_n = 1)	\\
&=~ C \sum_{n \geq 1} n^{1+\delta} \Prob(\sigma=n)   ~<~ \infty.
\end{align*}

Finally, for the proof of (e), assume that, for some $\delta > 0$,
Condition \ref{Con:Rate of convergence Z} is valid.
Proposition \ref{Prop:martingale moments} then implies that $\E [\oneW (\log^+ \oneW)^{\delta}] < \infty$
and thereupon $\E [V (\log^+ V)^{\delta}] < \infty$
because $V=\oneW$ $\Prob$-a.s.~by part (a).
Hence $\E [V_1 (\log^+ V_1)^{\delta+1}] < \infty$ by Theorem 1.2 in \cite{Alsmeyer+Iksanov:2009}.
\end{proof}

\begin{Lemma}   \label{Lem:phi_J inherits}
For $\phi: \Omega \times \R \to [0,\infty)$ a
product-measurable, separable random characteristic, define, for
$t \in \R$,
\begin{equation}    \label{eq:phi_J}
\phi_{\J}(t) ~:=~   \sum_{x \prec \J} [\phi]_x(t-S(x)).
\end{equation}
The following assertions hold.
\begin{itemize}
\item[(a)] Suppose that $\E^i \! \big[\sup_{0 \leq s \leq t} \phi(s) \big] < \infty$ for $i=1,\ldots,p$ and all $t \geq 0$.
Then $\phi_{\J}$ is product-measurable, $\sup_{0 \leq s \leq t}
\phi_{\J}(s) \leq Y_t$ for random variables $Y_t$, $t \geq 0$ with
$\E [Y_t] < \infty$. Further, if $\phi$ has $D$-valued paths, so has $\phi_{\J}$.
\item[(b)] Suppose that $\E^i \! \big[\sup_{0 \leq s \leq t} \phi(s) \big] < \infty$ for all $t\geq 0$
and that $t \mapsto e^{-\alpha t} \E^i[\phi(t)]$ is directly Riemann integrable on $\R_{\geq 0}$, $i=1,\ldots,p$.
Then $t \mapsto e^{-\alpha t} \E[\phi_{\J}(t)]$ is directly Riemann
integrable on $\R_{\geq 0}$.
\item[(c)] If Conditions \ref{Con:SLLN mu} and \ref{Con:SLLN phi} hold with $g=h$, then
\begin{equation*}
\sup_{t \geq 0} \, (h(t) \vee 1) e^{-\alpha t} \phi_{\J}(t)
\end{equation*}
has finite expectation with respect to $\Prob$.
\item[(d)] Assume that Condition \ref{Con:Ratio convergence mu} is satisfied.
If Condition \ref{Con:Ratio convergence phi} holds for $\phi$,
then there exists a $\beta < \alpha$ such that $\E[\sup_{t \geq 0}
e^{-\beta t} \phi_{\J}(t)] < \infty$.
\item[(e)] If Condition \ref{Con:Rate of convergence phi}
holds for $\phi$, then it also holds for $\phi_{\J}$.
\end{itemize}
\end{Lemma}
\begin{proof}
(a)
From the representation
\begin{equation*}
\phi_{\J}(t) ~=~    \sum_{x \in \I} \1_{\{x \in \G\}} \1_{\{x \prec \J\}} [\phi]_x(t-S(x))
\end{equation*}
it can be concluded that $\phi_{\J}$ is product-measurable. In
fact, it suffices to check that each summand is
product-measurable. Fix any $x\in \I$. The factors
$\1_{\{x \in \G\}}$ and $\1_{\{x \prec \J\}}$ are $\A$-measurable.
Since they do not depend on $t$, they are product-measurable. The
shift $\omega \mapsto \sigma_x \omega$ is measurable, thus the
mapping $(\omega,t) \mapsto (\sigma_x \omega,t)$ is
product-measurable and hence so is $[\phi]_x$. Finally, since
$(\omega,t) \mapsto t-S(x,\omega)$ is product-measurable, so is
$[\phi]_x(t-S(x))$. Further,
\begin{equation*}
\sup_{0 \leq s \leq t} \phi_{\J}(s) ~\leq~ Y_t ~:=~ e^{\alpha t} \sum_{x \prec \J}  e^{-\alpha S(x)} \sup_{0 \leq s \leq t}  [\phi]_x(s)
\end{equation*}
where it should be recalled that $\sum_{x \prec \J}$ means summation over the $x \in \G$ with $x \prec \J$.
$Y_t$ is a random variable since $\phi$ is separable.
In view of \eqref{eq:sigma<->prec J general},
\begin{eqnarray}
\E [Y_t]
& = &
e^{\alpha t} \E \!\bigg[\sum_{x \prec \J}  e^{-\alpha S(x)} \sup_{0 \leq s \leq t}  [\phi]_x(s) \bigg]
~=~ e^{\alpha t} \E \!\bigg[\sum_{k=0}^{\sigma-1} \frac{v_1}{v_{M_k}} \E^{M_k} \Big[ \sup_{0 \leq s \leq t}  \phi(s) \Big]\bigg]    \notag  \\
& \leq &
e^{\alpha t} \, v_1 \, \E [\sigma] \max_{j=1,\ldots,p} \frac{\E^j \! \big[\sup_{0 \leq s \leq t} \phi(s) \big]}{v_j} ~<~    \infty. \label{eq:Esup_t phi_J<infty}
\end{eqnarray}
In particular, $\phi_{\J}(t)$ is finite for all $t \geq 0$ (simultaneously) a.s.
If $(x_k)_{k \geq 1}$ is an enumeration of $\I$, then $\phi_{\J}$ is the almost sure limit of
\begin{equation*}
\phi_n(t)   ~:=~    \sum_{k=1}^{n} \1_{\{x_k \in \G\}} \1_{\{x_k \prec \J\}}[\phi]_{x_k}(t-S_k).
\end{equation*}
What is more, the almost sure convergence of $\phi_n$ to $\phi_{\J}$ is locally uniform since
\begin{equation*}
\sup_{0 \leq s \leq t} \big|\phi_n(s)-\phi_{\J}(s) \big|
~\leq~ \sum_{k>n} \1_{\{x_k \in \G\}} \1_{\{x_k \prec \J\}} \sup_{0 \leq s \leq t}  [\phi]_{x_k}(s) ~\leq~  Y_t
\end{equation*}
and
\begin{equation}
\E \bigg[\sum_{k>n}  \1_{\{x_k \in \G\}} \1_{\{x_k \prec \J\}} \sup_{0 \leq s \leq t}  [\phi]_{x_k}(s) \bigg] ~\to~   0   \quad   \text{as } n
\to \infty \label{eq:a.s. uniform convergence}
\end{equation}
by the dominated convergence theorem.
Hence, if $\phi$ is $D$-valued, so is $\phi_{\J}$, as the locally uniform limit of $D$-valued functions.

\noindent (b) \eqref{eq:a.s. uniform convergence} implies that
$\overline{\phi}_n = \E[\phi_n]$ converges locally uniformly to
$\overline{\phi}_{\J} = \E[\phi_{\J}]$. In particular,
$\overline{\phi}_{\J}$ is continuous at each point in
which all $\overline{\phi}_n$ are continuous. Consequently,
$\overline{\phi}_{\J}$ is continuous almost everywhere with
respect to Lebesgue measure.
Using \eqref{eq:sigma<->prec J general} we obtain
\begin{eqnarray}
e^{-\alpha t} \, \overline{\phi}_{\J}(t)
& = &
e^{-\alpha t} \, \E \bigg[\sum_{x \prec \J} [\phi]_x(t-S(x)) \bigg] \notag  \\
& = &
\E \bigg[\sum_{x \prec \J} e^{-\alpha S(x)} e^{-\alpha(t-S(x))} [\phi]_x(t-S(x)) \bigg] \notag  \\
& = & \E \bigg[\sum_{k=0}^{\sigma-1} \frac{v_1}{v_{M_k}}
e^{-\alpha(t-S_k)} \,^{M_k}\overline{\phi}(t-S_k) \bigg]	\notag	\\
&\leq &
\E \bigg[\sum_{k=0}^{\sigma-1} \max_{i=1,\ldots,p} \frac{v_1}{v_i}e^{-\alpha(t-S_k)} \,^i \overline{\phi}(t-S_k) \bigg]	\label{eq:e^-alpha t E phi_J(t)}
\end{eqnarray}
where $\iphi(t) = \E^i[\phi(t)]$, $i=1,\ldots,p$. The
latter implies that since $t\mapsto \max_{i=1,\ldots,p} \frac{v_1}{v_i}e^{-\alpha t} \,^i \overline{\phi}(t)$ is bounded
(as directly Riemann integrable), so is $t\mapsto e^{-\alpha t} \, \overline{\phi}_{\J}(t)$.
Set
\begin{equation*}
A	~:=~	\sum_{k\geq 0}\sup_{k \leq t<k+1} \max_{i=1,\ldots,p} \frac{v_1}{v_i}e^{-\alpha t} \,^i \overline{\phi}(t)
\end{equation*}
and note that $A<\infty$ because of direct Riemann integrability.
Using \eqref{eq:e^-alpha t E phi_J(t)} we infer
\begin{equation*}
\sum_{k\geq 0}\sup_{k\leq t<k+1}e^{-\alpha t} \, \overline{\phi}_{\J}(t)\leq A\E [\sigma]<\infty,
\end{equation*}
and according to Remark 3.10.4 in p.~236 in \cite{Resnick:2002} the direct
Riemann integrability of $t\mapsto e^{-\alpha t} \, \overline{\phi}_{\J}(t)$ follows.

\noindent (c)
Condition \ref{Con:SLLN phi} guarantees that
$U := \sup_{t \geq 0}  \, (h(t) \vee 1) e^{-\alpha t}\phi(t)$ has finite expectation w.r.t.\ $\Prob^i$ for $i=1,\ldots,p$.
Further,
\begin{eqnarray*}
\sup_{t \geq 0} \, (h(t) \vee 1) e^{-\alpha t} \phi_{\J}(t)
& \leq &
\sum_{x \prec \J} e^{-\alpha S(x)} \sup_{t \geq 0} \, (h(t) \vee 1) e^{-\alpha (t-S(x))} [\phi]_x(t-S(x))   \\
& \leq &
\sum_{x \prec \J} e^{-\alpha S(x)} [U]_x \sup_{t \geq 0} \Big( \1_{\{S(x) \leq t\}}  \frac{h(t) \vee 1}{h(t-S(x)) \vee 1}\Big). \\
\end{eqnarray*}
Notice that due to the fact that $h$ is increasing and regularly varying of index $1$ at $+\infty$,
we can find a finite constant $c \geq 1$ such that, for all $t \geq 0$,
$(h(t) \vee 1)/(h(t-S(x)) \vee 1) \leq c$ when $S(x) \leq t/2$.
On the other hand, for all $t \geq 0$,
when $S(x) \geq t/2$, then $(h(t) \vee 1)/(h(t-S(x)) \vee 1) \leq h(2S(x)) \vee 1$.
Consequently,
\begin{equation*}
\sup_{t \geq 0} \Big( \1_{\{S(x) \leq t\}}  \frac{h(t) \vee 1}{h(t-S(x)) \vee 1} \Big) \leq h(2S(x)) \vee c.
\end{equation*}
Plugging this into the estimate above and integrating w.r.t.\ $\Prob$ gives
\begin{eqnarray*}
\E \Big[\sup_{t \geq 0} \, (h(t) \vee 1) e^{-\alpha t} \phi_{\J}(t)\Big]
& \leq &
\E[U] \, \E\bigg[\sum_{x \prec \J} e^{-\alpha S(x)} (h(2S(x)) \vee c)\Big]  \\
& \leq &
\E[U] \, \max_{i=1,\ldots,p} \frac{v_1}{v_i} \,\E\bigg[\sum_{x \prec \J} \frac{v_{\tau(x)}}{v_1}e^{-\alpha S(x)} (h(2S(x)) \vee c)\Big] \\
& \leq & \E[U] \, \max_{i=1,\ldots,p} \frac{v_1}{v_i}
\,\E\bigg[\sum_{k=0}^{\sigma-1} (h(2S_k) \vee c) \bigg]
\end{eqnarray*}
with $\sigma = \inf\{k > 0: M_k = 1\}$.
Finiteness of the last expectation is a consequence of Lemma \ref{Lem:moments MRP}
which applies due to Remark \ref{Rem:function h}.

\noindent (d)
Condition \ref{Con:Ratio convergence phi} implies that
there exists $\theta < \alpha$ such that
\begin{equation*}
C_\theta:=v_1\max_{i=1,\ldots,p} {\E^i \Big[\sup_{t \geq 0}
e^{-\theta t} \phi(t)\Big]\over v_i} < \infty.
\end{equation*}
By Lemma \ref{Lem:M}(c) there are $C, \gamma > 0$ such that
$\Prob(\sigma > n) \leq C^2 e^{-\gamma n}$ for all $n \geq 0$.
Condition \ref{Con:Ratio convergence mu} ensures that
$c(2(\alpha-\beta))= \max_{i=1,\ldots,p} \E^i[e^{2(\alpha-\beta)S_1}] < e^{\gamma}$ for some $\beta \in
(\theta,\alpha)$ (see the proof of Proposition \ref{Prop:J
inherits2}{b}). Now the claim follows from
\begin{align*}
\E \Big[ & \sup_{t \geq 0} e^{-\beta t} \phi_{\J}(t) \Big]
~=~ \E \bigg[ \sup_{t \geq 0} e^{-\beta t} \sum_{x \prec \J} [\phi]_x(t-S(x)) \bigg]    \\
& \leq~
\E \bigg[  \sum_{x \prec \J} e^{-\beta S(x)} \sup_{t \geq 0} e^{-\beta(t-S(x))} [\phi]_x(t-S(x)) \bigg] \\
& \leq~
C_{\beta} \E \bigg[  \sum_{k=0}^{\sigma-1} e^{(\alpha-\beta) S_k} \bigg] \\
& =~
C_{\beta} \sum_{k \geq 0} \E \big[\1_{\{\sigma > k\}} e^{(\alpha-\beta) S_k} \big]   \\
& \leq~
C_{\beta} \sum_{k \geq 0} \Prob(\sigma > k)^{1/2} \, \E [e^{2(\alpha-\beta) S_k}]^{1/2} \\
& \leq~ C_{\beta}  C \sum_{k \geq 0} e^{-\gamma k/2}
c(2(\alpha-\beta))^{k/2} ~<~ \infty,
\end{align*}
where the Cauchy-Schwarz inequality has been used for the $5$th line.

\noindent
(e)
While the Lebesgue integrability of $t \mapsto e^{-\alpha t} \E[\phi_{\J}(t)]$
follows from
\begin{equation}	\label{eq:e^-alpha t E phi_J(t) Lebesgue integrable}
\int_0^\infty e^{-\alpha t} \, \overline{\phi}_{\J}(t) \, \dt
~\leq~ \E[\sigma] \max_{i=1,\ldots,p} \frac{v_1}{v_i} \int_0^\infty e^{-\alpha t} \iphi(t) \, \dt ~<~ \infty
\end{equation}
which is a consequence of \eqref{eq:e^-alpha t E phi_J(t)} and the
integrability of $t\mapsto \max_{i=1,\ldots,p} \frac{v_1}{v_i} e^{-\alpha t} \iphi(t)$,
its boundedness follows from the boundedness of $t\mapsto \max_{i=1,\ldots,p} \frac{v_1}{v_i}
e^{-\alpha t} \iphi(t)$, \eqref{eq:e^-alpha t E phi_J(t)} and $\E[\sigma] < \infty$ (see Lemma \ref{Lem:M}(b)).

It remains to show that \eqref{eq:Rate of convergence phi} holds for $\phi_{\J}$, that is,
\begin{equation*}
t^{\delta} \int_t^{\infty} e^{-\alpha s} \E[\phi_{\J}(s)] \,\ds ~\to~   0
\quad   \text{and}  \quad
t^{\delta} \sup_{s \geq t} e^{-\alpha s} \E [\phi_{\J}(s)]  ~\to~   0
\quad   \text{as } t \to \infty.
\end{equation*}
As to the first relation, notice that by a variant of the argument leading to \eqref{eq:e^-alpha t E phi_J(t) Lebesgue integrable},
\eqref{eq:Rate of convergence phi} and the dominated convergence theorem,
we have
\begin{align*}
t^{\delta} \int_t^{\infty} & e^{-\alpha s} \E[\phi_{\J}(s)] \,\ds
~\leq~ \E\bigg[\sum_{k=0}^{\sigma-1} v_1 t^{\delta} \max_{i=1,\ldots,p}  \int_{t-S_k}^{\infty} e^{-\alpha s} \frac{\iphi(s)}{v_i} \bigg] \,\ds
~\to~   0
\end{align*}
as $t \to \infty$. Similarly, the second relation holds since
\begin{align*}
t^{\delta} \sup_{s \geq t} & \, e^{-\alpha s} \E [\phi_{\J}(s)]
~\leq~  v_1 t^{\delta} \sup_{s \geq t} \E\bigg[\sum_{k=0}^{\sigma-1} e^{-\alpha (s-S_k)} \max_{i=1,\ldots,p} \frac{\iphi(s-S_k)}{v_i} \bigg]    \\
& \leq~ v_1 \E\bigg[\sum_{k=0}^{\sigma-1} \max_{i=1,\ldots,p} t^{\delta} \frac{\sup_{s \geq t-S_k} e^{-\alpha s}  \iphi(s)}{v_i} \bigg]
~\to~   0   \quad   \text{as } t \to \infty
\end{align*}
by the dominated convergence theorem (using that $\sup_{t \geq 0} e^{-\alpha t} \iphi(t) < \infty$ for $i=1,\ldots,p$).
\end{proof}

\section{Proofs of the main results}    \label{sec:proofs}


The proofs of the main results rely on a decomposition of $\Z^{\phi}(t)$ along the optional lines $\J_n$, $n \geq 0$.
For a random characteristic $\psi$, define $\Z^{1,\psi}$ by
\begin{equation}    \label{eq:M^1,psi}
\Z^{1,\psi}(t)  ~:=~    \sum_{x \in \G^1} [\psi]_x(t-S(x))
\end{equation}
where $\G^1 := \{x \in \G: \tau(x) = 1\}$.
We choose $\psi := \phi_{\J}$ as defined in \eqref{eq:phi_J}.
Then
\begin{eqnarray}
\Z^{1,\phi_{\J}}(t)
& = &   \sum_{x \in \G^1} [\phi_{\J}]_x(t-S(x))
~=~ \sum_{x \in \G^1} \sum_{y \prec [\J]_x} [\phi]_{xy}(t-S(xy))    \notag  \\
& = &   \sum_{x \in \G} [\phi]_x(t-S(x))
~=~ \Z^{\phi}(t).   \label{eq:key identity}
\end{eqnarray}
Using this connection, limit theorems for the multi-type process
will be derived from the corresponding single-type
ones\footnote{In \cite{Meiners:2010}, along similar
lines limit theorems for branching random walks on the line are
obtained from the corresponding results for
branching random walks with positive steps only.}.

\subsection{Convergence in probability}

\begin{proof}[Proof of Theorem \ref{Thm:WLLN}]
By \eqref{eq:key identity}, the first part of the theorem
follows if we can check that $\Z_{\J}$ and $\phi_{\J}$ satisfy the
assumptions of Theorem 3.1 in \cite{Nerman:1981}. That $\Z_{\J}$
satisfies the standing assumptions given in p.~366 of
\cite{Nerman:1981} is established in Proposition \ref{Prop:J
inherits1}. That $\phi_{\J}$ satisfies the conditions of Theorem
3.1 in \cite{Nerman:1981} is secured by Lemma \ref{Lem:phi_J
inherits} with the exception that it cannot be guaranteed that
$\phi_{\J}$ is separable. On the other hand, the perusal of the
proof of Theorem 3.1 in \cite{Nerman:1981} makes it clear that the
assumption of separability can be omitted as long as $\sup_{s \leq
t} \phi_{\J}(s)$ is dominated by an integrable random variable.
This is indeed the case here, see Lemma \ref{Lem:phi_J
inherits}(a). Consequently, Theorem 3.1 in \cite{Nerman:1981}
yields
\begin{equation}
e^{-\alpha t} \Z^{\phi}(t)  ~=~ e^{-\alpha t} \Z^{1,\phi_{\J}}(t)   ~\to~  \frac{W}{-m'(\alpha)} \int_0^{\infty} e^{-\alpha t} \E [\phi_{\J}(t)] \, \dt
\end{equation}
in probability as $t \to \infty$. From Proposition \ref{Prop:J
inherits1}(c), we know that
\begin{equation*}
-m'(\alpha) = (u_1 v_1)^{-1} \sum_{i,j=1}^p u_i v_j (-\im_j)'(\alpha).
\end{equation*}
Left with the calculation of the integral, we write, recalling \eqref{eq:phi_J}
and using \eqref{eq:sigma<->J general} and Lemma \ref{Lem:M}(b),
\begin{align*}
\int_0^{\infty} & e^{-\alpha t} \E [\phi_{\J}(t)] \, \dt
~=~
\int_0^{\infty} \E \bigg[\sum_{x \prec \J} e^{-\alpha S(x)} e^{-\alpha(t-S(x))}[\phi]_x(t-S(x))\bigg] \, \dt    \\
& =~
\int_0^{\infty} \E \bigg[\sum_{k=0}^{\sigma-1} e^{-\alpha(t-S_k)} \frac{v_1}{v_{M_k}}\cdot \,\!^{M_k}\overline{\phi}(t-S_k) \bigg] \, \dt   \\
& =~
\sum_{i=1}^p \E [\#\{k < \sigma: M_k=i\}] \frac{v_1}{v_i} \int_0^{\infty}  e^{-\alpha t} \E^i[\phi(t)] \, \dt   \\
& =~
\sum_{i=1}^p \frac{u_i}{u_1} \int_0^{\infty}  e^{-\alpha t} \E^i[\phi(t)] \, \dt
\end{align*}
where $^i\overline{\phi}(t) = \E^i[\phi(t)]$.

As for convergence in mean, observe that (A5) implies
$\E[V_1\log^+V_1]<\infty$ by Proposition \ref{Prop:J inherits2}
(a) and apply Corollary 3.3 in \cite{Nerman:1981}.
\end{proof}

\subsection{Almost sure convergence}
\begin{proof}[Proof of Theorem \ref{Thm:SLLN}]
The proof is similar to the proof of Theorem \ref{Thm:WLLN}
but based on an application of Theorem 5.4 in \cite{Nerman:1981}
rather than Theorem 3.1 in the same source.
It suffices to check that the conditions of Theorem 5.4 in \cite{Nerman:1981}
are satisfied.
As in the proof of Theorem \ref{Thm:WLLN}
we conclude from Proposition \ref{Prop:J inherits1} that $\Z_{\J}$ satisfies the standing assumptions given in p.\;366 of \cite{Nerman:1981}.
Further, $\phi_{\J}$ has $D$-valued paths by Lemma \ref{Lem:phi_J inherits}(a).
The lemma applies because Condition \ref{Con:SLLN phi} for $\phi$ entails
$\E^i [\sup_{0 \leq s \leq t} \phi(s)] < \infty$ for $i=1,\ldots,p$ and all $t\geq 0$.
It remains to check Conditions 5.1 and 5.2 of \cite{Nerman:1981} for the embedded single-type process and the characteristic $\phi_{\J}$, respectively.
To see that Condition 5.1 holds in the present context first notice that \cite[Eq.\;(5.8)]{Nerman:1981}
is sufficient for Condition 5.1 and that validity of \cite[Eq.\;(5.8)]{Nerman:1981}
follows from Condition \ref{Con:SLLN mu} together with Proposition \ref{Prop:J inherits2}(c).
Finally, $\phi_{\J}$ satisfies Condition 5.2 of \cite{Nerman:1981} by Lemma \ref{Lem:phi_J inherits}(c)
(notice that one can assume without loss of generality that $g=h$ in Conditions \ref{Con:SLLN mu} and \ref{Con:SLLN phi}).
The form of the limit can be deduced from Theorem \ref{Thm:WLLN}.
\end{proof}

\subsection{Ratio convergence}

\begin{proof}[Proof of Theorem \ref{Thm:Ratio convergence}]
The scheme of proof is identical to that of Theorem \ref{Thm:WLLN}
but it is based on an application of Theorem 6.3 in \cite{Nerman:1981} rather than Theorem 3.1 in the same source.
Hence, we have to check that the assumptions of Theorem 6.3 in \cite{Nerman:1981} are fulfilled.
Since Condition \ref{Con:Ratio convergence mu} holds, we
conclude that $m(\beta) < \infty$ for some $\beta < \alpha$ by
Proposition \ref{Prop:J inherits2}(b). Further $\phi_{\J}$ and
$\psi_{\J}$ have $D$-valued paths by Lemma \ref{Lem:phi_J
inherits}(a) which applies because Condition \ref{Con:Ratio
convergence phi} for $\phi$ and $\psi$ entails $\E^i \!
\big[\sup_{0 \leq s \leq t} \phi(s) \big] < \infty$ and $\E^i \!
\big[\sup_{0 \leq s \leq t} \psi(s) \big] < \infty$ for
$i=1,\ldots,p$ and all $t\geq 0$. Finally, invoking Lemma
\ref{Lem:phi_J inherits}(d) gives $\E[\sup_{t \geq 0} e^{-\beta
t} \phi_{\J}(t)] < \infty$ and $\E[\sup_{t \geq 0} e^{-\beta t}
\psi_{\J}(t)] < \infty$, where we assume w.l.o.g.~that the $\beta$
from Lemma \ref{Lem:phi_J inherits} is the same $\beta$ for which
$m(\beta) < \infty$. Theorem 6.3 in \cite{Nerman:1981} now yields
that, on $S$,
\begin{equation*}
\frac{\Z^{\phi}(t)}{\Z^{\psi}(t)}   ~\to~   \frac{\int_0^{\infty} e^{-\alpha t} \E [\phi_{\J}(t)] \, \dt}{\int_0^{\infty} e^{-\alpha t} \E [\psi_{\J}(t)] \, \dt}
\quad   \text{a.s.~as } t \to \infty.
\end{equation*}
Numerator and denominator of the fraction on the right-hand side have been calculated in the proof of Theorem \ref{Thm:WLLN}.
\end{proof}

\subsection{Rate of convergence}

We first prove Theorem \ref{Thm:Rates of convergence} in the case
$p=1$, that is, in the single-type case in which $\bu = \bv = 1$. Recall the definition of $\J(t)$:
\begin{equation*}
\J(t) := \{x \in \G: S(x) > t,\,S(x|_k) \leq t \text{ for all } k < |x|\}.
\end{equation*}

\begin{proof}[Proof of Theorem \ref{Thm:Rates of convergence}: The single-type case]
We have to show that
\begin{equation}    \tag{\ref{eq:Rate in WLLN}}
t^{\delta}|e^{-\alpha t} \Z^{\phi}(t)-Vm^{\phi}_{\infty}| \to 0 \quad   \text{in } \Prob\text{-probability as } t \to \infty
\end{equation}
where $m^{\phi}_t := e^{-\alpha t} \E[\Z^{\phi}(t)]$ and
\begin{equation*}
m^{\phi}_{\infty}   ~=~ \lim_{t \to \infty} m^{\phi}_t      ~=~ \frac{1}{-m'(\alpha)} \int_0^{\infty} e^{-\alpha u} \E[\phi(u)] \, \du.
\end{equation*}
This limit exists since $m^{\phi}_{t}$ solves a renewal equation,
see formula (2.4) in \cite{Nerman:1981}. To prove
\eqref{eq:Rate in WLLN}, we truncate $\phi$ at some $c > 0$ and
set $\phi_c(t) := \phi(t) \1_{[0,c]}(t)$, $t \geq 0$. For
$0 < c < s \leq t$, we consider \eqref{eq:Rate in WLLN} at $t+s$
and use the triangular inequality to obtain
\begin{align*}
(t+s)^{\delta} & |e^{-\alpha (t+s)} \Z^{\phi}(t+s)-V m^{\phi}_{\infty}| \\
& \leq~ (t+s)^{\delta} \big(e^{-\alpha (t+s)} (\Z^{\phi}(t+s) - \Z^{\phi_c}(t+s)) + V(m^{\phi}_{\infty}-m^{\phi_c}_{\infty})\big)  \\
&\hphantom{\leq}~ + (t+s)^{\delta} |e^{-\alpha (t+s)} \Z^{\phi_c}(t+s)-V m^{\phi_c}_{\infty}|.
\end{align*}
The expectation of the first summand on the right-hand side is equal to
\begin{eqnarray*}
(t+s)^{\delta} \big(m^{\phi}_{t+s} - m^{\phi_c}_{t+s} + m^{\phi}_{\infty}-m^{\phi_c}_{\infty}\big)
& = &
(t+s)^{\delta} \big(m^{\phi-\phi_c}_{t+s} + m^{\phi-\phi_c}_{\infty} \big)  \\
& \leq &
(t+s)^{\delta} \big(|m^{\phi-\phi_c}_{t+s}-m^{\phi-\phi_c}_{\infty}| +  2 m^{\phi-\phi_c}_{\infty} \big).
\end{eqnarray*}
Here, when choosing $2c=s=t$, we have
\begin{equation*}
(t+s)^{\delta} m^{\phi-\phi_c}_{\infty} ~=~ \frac{(t+s)^{\delta}}{-m'(\alpha)} \int_c^{\infty} e^{-\alpha u} \E[\phi(u)] \, \du
~\to~   0   \quad   \text{as } t \to \infty
\end{equation*}
by \eqref{eq:Rate of convergence phi}. Assumption
\eqref{eq:ES_1^(1+delta)<infty} is equivalent to
$\E[S_1^{1+\delta}] < \infty$. This together with Condition
\ref{Con:spread-out} (see, in particular, Remark
\ref{Rem:spread-out}) and Condition \ref{Con:Rate of convergence
phi} allows us to apply Theorem 4.2(ii) in
\cite{Nummelin+Tuominen:1983} (with $\psi(t)=t^\delta$ and
$g(s)=e^{-\alpha s} \E[\phi(s)]$, the notation of
\cite{Nummelin+Tuominen:1983}) which gives
\begin{eqnarray*}
(t+s)^{\delta} |m^{\phi-\phi_c}_{t+s}-m^{\phi-\phi_c}_{\infty}|
& \leq &
(t+s)^{\delta} \sup_{f} \bigg|f*U(t+s)-\frac{1}{-m'(\alpha)} \int_0^{\infty} f(u) \,\du \bigg|
~\to~	0
\end{eqnarray*}
as $t \to \infty$
where $U$ denotes the renewal measure of the random walk $(S_n)_{n \geq 0}$ and
the supremum is over all Lebesgue integrable functions $f \geq 0$ satisfying $f(u) \leq e^{-\alpha u} \E[\phi(u)]$ for all $u \in \R$.
It thus remains to show that
\begin{equation}
(t+s)^{\delta} |e^{-\alpha (t+s)} \Z^{\phi_c}(t+s) - V m^{\phi_c}_{\infty}| ~\stackrel{\Prob}{\to}~ 0   \quad   \text{as } t \to \infty
\end{equation}
where $2c=s=t$. To this end, we choose $0 < a < 1$ and estimate as
follows
\begin{align*}
(t&+s)^{\delta} |e^{-\alpha(t+s)} \Z^{\phi_c}(t+s) - V m^{\phi_c}_{\infty}| \\
&\leq~  (t+s)^{\delta} \bigg| \! \sum_{x \in \J(t)} e^{-\alpha S(x)} (e^{-\alpha(t+s-S(x))} [\Z^{\phi_c}]_x(t\!+\!s\!-\!S(x))-m^{\phi_c}_{t+s-S(x)})\bigg|  \\
&\hphantom{\leq}~   + (t+s)^{\delta} \sum_{x \in \J(t):S(x) \leq t+at} \!\!\!\!\!\!\!\!\! e^{-\alpha S(x)} |m^{\phi_c}_{t+s-S(x)}-m^{\phi_c}_{\infty}|  \\
&\hphantom{\leq}~   + (t+s)^{\delta} \sum_{x \in \J(t):S(x) > t+at} \!\!\!\!\!\!\!\!\! e^{-\alpha S(x)} |m^{\phi_c}_{t+s-S(x)}-m^{\phi_c}_{\infty}| \\
&\hphantom{\leq}~   + (t+s)^{\delta} |V(t)-V| m^{\phi_c}_{\infty}   \\
&=:~  \sum_{j=1}^4 |I_j(t)|.
\end{align*}
Since \eqref{eq:T_te^alphat sufficient} is a consequence of
\eqref{eq:ES_1^(1+delta)<infty} and Condition \ref{Con:Rate of
convergence Z} holds, Proposition \ref{Prop:rate of convergence of
Nerman's martingale} applies and gives
\begin{equation*}
|I_4(t)| ~=~    (2t)^{\delta} |V(t)-V| m^{\phi_c}_{\infty}  ~\to~   0   \quad   \text{a.s.~as } t \to \infty.
\end{equation*}
Further,
\begin{eqnarray*}
\E [|I_3(t)|]
& \leq &
2 \Big( \sup_{u \geq 0} m^{\phi_c}_u \Big)(t+s)^{\delta} \, \E \bigg[ \sum_{x \in \J(t):S(x) > t+at} \!\!\!\!\!\!\!\!\! e^{-\alpha S(x)}\bigg]  \\
& \leq &
2 \Big( \sup_{u \geq 0} m^{\phi}_u \Big)(2t)^{\delta} \, \Prob(R_t > at)
\end{eqnarray*}
where $R_t = S_{\nu(t)}-t$ and $\nu(t) = \inf\{n \in \N_0: S_n > t\}$.
An application of Lemma \ref{Lem:speed of excess} yields $\E[|I_3(t)|] \to 0$ when $t \to \infty$.
Turning to $I_2$, we see that
\begin{eqnarray*}
\E[|I_2(t)|]
& = &
\E \bigg[\sum_{x \in \J(t):S(x) \leq t+at} \!\!\!\!\!\!\!\!\! e^{-\alpha S(x)} (t+s)^{\delta} |m^{\phi_c}_{t+s-S(x)}-m^{\phi_c}_{\infty}|\bigg] \\
& \leq &
(2t)^{\delta} \sup_{u \geq (1-a)t} |m^{\phi_c}_{u}-m^{\phi_c}_{\infty}|
~\to~   0
\end{eqnarray*}
as $t \to \infty$ by Theorem 4.2(ii) in
\cite{Nummelin+Tuominen:1983} (applicability of the cited
result has already been justified).

It remains to show that $I_1(t) \to 0$ in $\Prob$-probability when
$t \to \infty$. For fixed $t$, define $Z_x :=
e^{-\alpha(t+s-S(x))} [\Z^{\phi_c}]_x(t\!+\!s\!-\!S(x))$. Then,
given $\H_{T_t}$, the $Z_x$, $x \in \J(t)$ are independent. In
order to use Chebyshev's inequality below, we truncate the $Z_x$.
Define $Z_x' := Z_x \1_{\{Z_x \leq e^{\alpha S(x)}  \}}$ and $m'_x
:= \E [Z_x']$. Notice that $Z_x' \leq Z_x$ and hence $m_x' \leq
m^{\phi_c}_{t+s-S(x)}$. Let $I_1'(t)$ be defined as $I_1(t)$
but with $Z_x$ replaced by $Z_x'$ and $m^{\phi_c}_{t+s-S(x)}$
replaced by $m'_x$. Then
\begin{equation*}
I_1(t) ~=~    I_1'(t) + (t+s)^{\delta} \sum_{x \in \J(t)}
e^{-\alpha S(x)} (m'_x-m^{\phi_c}_{t+s-S(x)})
\end{equation*}
on $\{Z_x \leq e^{\alpha S(x)} \text{ for all } x \in \J(t)\}$. Using this, we infer for arbitrary $\eta > 0$,
\begin{eqnarray*}
\Prob(|I_1(t)| \geq \eta)
& = &
\E[\Prob(|I_1(t)| \geq \eta | \H_{T_t})]    \\
&\leq &
\E \bigg[\sum_{x \in \J(t)} \Prob(Z_x > e^{\alpha S(x)}  | \H_{T_t})\bigg]  \\
& &
+ \E [\Prob(|I_1'(t)| \geq \eta/2 | \H_{T_t})]  \\
& &
+ \frac{2(t+s)^{\delta}}{\eta} \E \bigg[ \! \sum_{x \in \J(t)}\!\!\! e^{-\alpha S(x)} (m^{\phi_c}_{t+s-S(x)}-m_x')\bigg].
\end{eqnarray*}
We consider the last three terms separately.
Using Markov's inequality, we obtain for the first term that
\begin{align*}
\E & \bigg[ \sum_{x \in \J(t)} \Prob(Z_x > e^{\alpha S(x)}  | \H_{T_t})\bigg]   \\
&\leq~  \E \bigg[ \sum_{x \in \J(t)} e^{-\alpha S(x)}  \, \E[Z_x \1_{\{Z_x > e^{\alpha S(x)} \}} | \H_{T_t}]\bigg]  \\
&\leq~  \E \bigg[ \sum_{x \in \J(t)} e^{-\alpha S(x)}  \bigg] \sup_{u \geq 0} \E\big[e^{-\alpha u} \Z^{\phi}(u) \1_{\{e^{-\alpha u} \Z^{\phi}(u) > e^{\alpha t}\}}\big] \\
&=~  \sup_{u \geq 0} \E\big[e^{-\alpha u} \Z^{\phi}(u) \1_{\{e^{-\alpha u} \Z^{\phi}(u) > e^{\alpha t}\}}\big]
~\to~   0
\end{align*}
as $t \to \infty$ by the uniform integrability of the family
$e^{-\alpha u} \Z^{\phi}(u)$, $u \geq 0$ (which is a consequence
of Condition \ref{Con:Rate of convergence Z^phi}). Next, we
estimate the second term
\begin{align*}
\E & [\Prob(|I_1'(t)| \geq \eta/2 | \H_{T_t})]  \\
& \leq~ \frac{4(t+s)^{2\delta}}{\eta^2} \E\bigg[\sum_{x \in \J(t)} e^{-2\alpha S(x)} \Var[Z_x'| \H_{T_t}]\bigg] \\
& \leq~ \frac{4(t+s)^{2\delta}}{\eta^2} \E\bigg[\sum_{x \in \J(t)} e^{-2\alpha S(x)} \E[Z_x^2 \1_{\{Z_x \leq e^{\alpha S(x)}\}}| \H_{T_t}]\bigg]    \\
& =~ \frac{4(t+s)^{2\delta}}{\eta^2} \E\bigg[\sum_{x \in \J(t)} e^{-2\alpha S(x)} \E\bigg[h(Z_x) \frac{Z_x^2}{h(Z_x)} \1_{\{Z_x \leq e^{\alpha S(x)}\}}| \H_{T_t} \bigg]\bigg]  \\
& \leq~ \frac{4(t+s)^{2\delta}}{\eta^2} \E\bigg[\sum_{x \in \J(t)} e^{-\alpha S(x)} \frac{e^{\alpha S(x)}}{h(e^{\alpha S(x)})} \bigg]
\sup_{u \geq 0} \E\big[h\big(e^{-\alpha u} \Z^{\phi}(u)\big)\big]   \\
& \leq~ \frac{4^{1+\delta}}{\eta^2}\frac{e^{\alpha t} t^{2\delta}}{h(e^{\alpha t})}
\sup_{u \geq 0} \E\big[h\big(e^{-\alpha u} \Z^{\phi}(u)\big)\big]       ~\to~   0   \quad   \text{as } t \to \infty
\end{align*}
where we have used the independence of $Z_x'$ and $\H_{T_t}$, Chebyshev's inequality given $\H_{T_t}$
and the facts that $t \mapsto t^2/h(t)$ and $t \mapsto t/h(t)$ are increasing and decreasing, respectively, for large $t$,
$t (\log t)^{2 \delta} / h(t) \to 0$ as $t \to \infty$, and the finiteness of the supremum
(see Condition \ref{Con:Rate of convergence Z^phi}).
Finally, the third term can be estimated as follows:
\begin{align*}
\frac{2(t+s)^{\delta}}{\eta} & \E \bigg[ \! \sum_{x \in \J(t)} e^{-\alpha S(x)} (m^{\phi_c}_{t+s-S(x)}-m_x')\bigg]  \\
& =~
\frac{2^{1+\delta} t^{\delta}}{\eta} \E \bigg[ \! \sum_{x \in \J(t)} e^{-\alpha S(x)} \E[Z_x \1_{\{Z_x > e^{\alpha S(x)}\}} | \H_{T_t}]\bigg]   \\
& =~    \frac{2^{1+\delta} t^{\delta}}{\eta} \E \bigg[ \! \sum_{x \in \J(t)} e^{-\alpha S(x)} \E\bigg[\frac{Z_x}{h(Z_x)} h(Z_x) \1_{\{Z_x > e^{\alpha S(x)}\}} | \H_{T_t}\bigg]\bigg]   \\
& \leq~ \frac{2^{1+\delta} t^{\delta}}{\eta} \E \bigg[ \! \sum_{x \in \J(t)} e^{-\alpha S(x)} \frac{e^{\alpha S(x)}}{h(e^{\alpha S(x)})} \bigg]
\sup_{u \geq 0} \E\big[h\big(e^{-\alpha u} \Z^{\phi}(u)\big)\big]   \\
& \leq~  \frac{2^{1+\delta} t^{\delta}}{\eta}
\frac{e^{\alpha t}}{h(e^{\alpha t})} \sup_{u \geq 0}
\E\big[h\big(e^{-\alpha u} \Z^{\phi}(u)\big)\big] ~\to~   0
\end{align*}
as $t \to \infty$ using the same argument as above.
\end{proof}

\begin{proof}[Proof of Theorem \ref{Thm:Rates of convergence}: The general (multi-type) case]
In view of the embedding technique and key identity \eqref{eq:key
identity}, it is enough to show that the embedded single-type
process $(\Z_{\J_n})_{n \geq 0}$ and the characteristic
$\phi_{\J}$ fulfill the assumptions of the single-type version of
this theorem. According to Proposition \ref{Prop:J
inherits1}, $(\Z_{\J_n})_{n \geq 0}$ satisfies the counterparts of the
standing assumptions (A1)-(A4).
Validity of \eqref{eq:ES_1^(1+delta)<infty} and of Condition
\ref{Con:Rate of convergence Z} for the embedded process follows
from Proposition \ref{Prop:J inherits2}(d) and (e),
respectively. Condition \ref{Con:spread-out} carries over immediately
to the embedded process. Condition \ref{Con:Rate of
convergence Z^phi} is a condition on $\Z^{\phi}$. It is thus
identical for the original and the embedded process (via
\eqref{eq:key identity}). Condition \ref{Con:Rate of convergence
phi} holds for the embedded process according to Lemma \ref{Lem:phi_J inherits}.
\end{proof}

\begin{appendix}
\section{Auxiliary results}
\begin{Lemma}   \label{Lem:moments MRP}
Let $((M_n,S_n))_{n \in \N_0}$ be a Markov renewal process with $\Prob(M_0=1,S_0=0)=1$
and where the driving chain $(M_n)_{n \in \N_0}$ is irreducible and aperiodic and has state space $\{1,\ldots,p\}$.
Further, let $\sigma := \inf\{n \in \N: M_n = 1\}$ and $h(x) = xf(x)$, $x\geq 0$
where $f:[0,\infty) \to [0,\infty)$ is an increasing subadditive function.
If $\E[h(S_1) | M_0 = i] < \infty$ for $i = 1,\ldots,p$, then $\E[h(S_{\sigma})] < \infty$
and $\E \big[ \sum_{k=0}^{\sigma-1} h(S_k)\big] < \infty$.
\end{Lemma}
\begin{proof}
Let $F_{i,j}(\cdot) := \Prob(S_1 \in \cdot | M_0=i, M_1=j)$,
$i,j=1,\ldots,p$ and assume that $X_{n,i,j}$, $n \in \N$,
$i,j=1,\ldots,p$ is a family of independent random variables with
$\Prob(X_{n,i,j} \in \cdot) = F_{i,j}(\cdot)$. Further, suppose
that the family of $X_{n,i,j}$ is independent of the driving chain
$(M_n)_{n \in \N_0}$. We assume without loss of generality that
$S_n = \sum_{k=1}^n X_{k,M_{k-1},M_k}$, $n \in \N_0$.
Now define $X_n := \max_{i,j=1,\ldots,p} X_{n,i,j}$, $n \in \N$ and $T_n := \sum_{k=1}^n
X_k$, $n \in \N_0$. $(T_n)_{n \in \N_0}$ is a zero-delayed renewal process with $S_n \leq T_n$ for all $n \in \N_0$.
What is more, $\E[h(T_1)] = \E[h(X_1)] \leq \sum_{i,j=1}^p \E[h(X_{1,i,j})] < \infty$.
Therefore, in view of the monotonicity of $h$, in order to prove the lemma,
it suffices to check that $\E[h(T_{\sigma})] < \infty$ and $\E \big[ \sum_{k=0}^{\sigma-1} h(T_k)\big] < \infty$.
Using subadditivity, we obtain
\begin{eqnarray}	\label{eq:estimate}
\E[h(T_n)]
& \leq &
\E\bigg[\sum_{k=1}^n X_k(f(X_k)+f(T_n-X_k))\bigg]
~=~	n \E[X_1f(X_1)] + \sum_{k=1}^n \E[X_k] \E[f(T_n-X_k)]	\notag	\\
& \leq &
n \E[h(X_1)] + n(n-1)\E[X_1] \E [f(X_1)]
~\leq~ n^2\E[h(X_1)].
\end{eqnarray}
It remains to observe that $(T_n)_{n \in \N_0}$ and $\sigma$ are
independent by construction and that $\sigma$ has some finite
exponential moments by irreducibility. Hence with the help of
\eqref{eq:estimate} we infer
\begin{equation*}
\E[h(T_{\sigma})] = \sum_{n \geq 1} \Prob(\sigma=n) \E[h(T_n)] < \infty
\quad	\text{and}	\quad
\E \bigg[\sum_{n=0}^{\sigma-1} h(T_n)\bigg] = \sum_{n \geq 0} \Prob(\sigma > n) \E[h(T_n)] < \infty.
\end{equation*}
\end{proof}

\begin{Rem} \label{Rem:function h}
Fix any $\varepsilon>0$ and notice that $f(x) := \log^{1+\varepsilon}(x+e^{\varepsilon})$,
$x \geq 0$ is concave with $f(0)=\varepsilon > 0$, hence subadditive.
Since $h(x) := x \log^{1+\varepsilon}(1+x) \sim x f(x)$ as $x \to +\infty$,
the lemma applies also for this particular choice of $h$.
\end{Rem}

\begin{Lemma}   \label{Lem:speed of excess}
Assume that $(S_n)_{n \geq 0}$ is a zero-delayed renewal process
and that $\E [S_1^{1+\delta}] < \infty$ for some $\delta > 0$. If
$\nu(t) := \inf\{n \in \N_0: S_n > t\}$ is the first passage time
into $(t,\infty)$ and $R_{t} = S_{\nu(t)}-t$ is the excess
at time $t$, then
\begin{equation*}
t^{\delta} \Prob(R_t> a t) ~\to~    0   \quad   \text{as } t \to \infty
\end{equation*}
for every $a > 0$.
\end{Lemma}
\begin{proof}
With $U$ denoting the renewal function, write
\begin{equation*}
t^\delta\Prob(R_t>at)
~=~ t^\delta\int_{[0,\,t]}\Prob(S_1>(a+1)t-y) \, U(\dy)
~\leq~ t^{1+\delta} \Prob(S_1>at)\frac{U(t)}{t}
\end{equation*}
and observe that
$\lim_{t \to \infty} t^{1+\delta}\Prob(S_1>at)=0$ in view of $\E
[S_1^{1+\delta}]<\infty$, while $\lim_{t \to \infty} t^{-1}U(t) = \E [S_1] < \infty$
by the elementary renewal theorem.
\end{proof}
\end{appendix}

\footnotesize
\noindent   {\bf Acknowledgements}  \quad
The authors thank Olle Nerman for helpful suggestions.

\bibliographystyle{abbrv}

\begin{thebibliography}{10}

\bibitem{Alsmeyer+Iksanov:2009}
G.~Alsmeyer and A.~Iksanov.
\newblock A log-type moment result for perpetuities and its application to
  martingales in supercritical branching random walks.
\newblock {\em Electron. J. Probab.}, 14:289--313 (electronic), 2009.

\bibitem{Alsmeyer+al:2009}
G.~Alsmeyer, A.~Iksanov, S.~Polotskiy, and U.~R{\"o}sler.
\newblock Exponential rate of {$L_p$}-convergence of intrinsic martingales in
  supercritical branching random walks.
\newblock {\em Theory Stoch. Process.}, 15(2):1--18, 2009.

\bibitem{Asmussen+Hering:1983}
S.~Asmussen and H.~Hering.
\newblock {\em Branching processes}, volume~3 of {\em Progress in Probability
  and Statistics}.
\newblock Birkh\"auser Boston Inc., Boston, MA, 1983.

\bibitem{Biggins:1977}
J.~D. Biggins.
\newblock Martingale convergence in the branching random walk.
\newblock {\em J. Appl. Probab.}, 14(1):25--37, 1977.

\bibitem{Biggins+Kyprianou:2004}
J.~D. Biggins and A.~E. Kyprianou.
\newblock Measure change in multitype branching.
\newblock {\em Adv. Appl. Probab.}, 36(2):544--581, 2004.

\bibitem{Gatzouras:2000}
D.~Gatzouras.
\newblock On the lattice case of an almost-sure renewal theorem for branching
  random walks.
\newblock {\em Adv. Appl. Probab.}, 32(3):720--737, 2000.

\bibitem{Iksanov:2006}
A.~Iksanov.
\newblock On the rate of convergence of a regular martingale related to a
  branching random walk.
\newblock {\em Ukra\"\i n. Mat. Zh.}, 58(3):326--342, 2006.

\bibitem{Iksanov+Meiners:2010a}
A.~M. Iksanov and M.~Meiners.
\newblock Exponential rate of almost-sure convergence of intrinsic martingales
  in supercritical branching random walks.
\newblock {\em J. Appl. Probab.}, 47(2):513--525, 2010.

\bibitem{Jagers:1989}
P.~Jagers.
\newblock General branching processes as {M}arkov fields.
\newblock {\em Stochastic Process. Appl.}, 32(2):183--212, 1989.

\bibitem{Kyprianou+Sani:2001}
A.~E. Kyprianou and A.~R. Sani.
\newblock Martingale convergence and the functional equation in the multi-type
  branching random walk.
\newblock {\em Bernoulli}, 7(4):593--604, 2001.

\bibitem{Lyons:1997}
R.~Lyons.
\newblock A simple path to {B}iggins' martingale convergence for branching
  random walk.
\newblock In {\em Classical and modern branching processes ({M}inneapolis,
  {MN}, 1994)}, volume~84 of {\em IMA Vol. Math. Appl.}, pages 217--221.
  Springer, New York, 1997.

\bibitem{Meiners:2010}
M.~Meiners.
\newblock An almost-sure renewal theorem for branching random walks on the
  line.
\newblock {\em J. Appl. Probab.}, 47(3):811--825, 2010.

\bibitem{Nerman:1979}
O.~Nerman.
\newblock {\em On the Convergence of Supercritical General Branching
  Processes}.
\newblock PhD thesis, Chalmers University of Technology and the University of
  G\"oteborg, 1979.

\bibitem{Nerman:1981}
O.~Nerman.
\newblock On the convergence of supercritical general ({C}-{M}-{J}) branching
  processes.
\newblock {\em Z. Wahrsch. Verw. Gebiete}, 57(3):365--395, 1981.

\bibitem{Nummelin+Tuominen:1983}
E.~Nummelin and P.~Tuominen.
\newblock The rate of convergence in {O}rey's theorem for {H}arris recurrent
  {M}arkov chains with applications to renewal theory.
\newblock {\em Stoch. Proc. Appl.}, 15(3):295--311, 1983.

\bibitem{Olofsson:2009}
P.~Olofsson.
\newblock Size-biased branching population measures and the multi-type {$x\log
  x$} condition.
\newblock {\em Bernoulli}, 15(4):1287--1304, 2009.

\bibitem{Pyke:1961}
R.~Pyke.
\newblock Markov renewal processes: definitions and preliminary properties.
\newblock {\em Ann. Math. Statist.}, 32:1231--1242, 1961.

\bibitem{Resnick:2002}
S.~I. Resnick.
\newblock {\em Adventures in stochastic processes}.
\newblock {B}irkh\"{a}user, Boston, 2002.

\bibitem{Shurenkov:1984}
V.~M. Shurenkov.
\newblock On {M}arkov renewal theory.
\newblock {\em Teor. Veroyatnost. i Primenen.}, 29(2):248--263, 1984.

\end{thebibliography}

\vspace{0.1cm}
\noindent
\textsc{Alexander Iksanov   \\
Faculty of Cybernetics  \\
National T.\ Shevchenko University of Kyiv,\\
01601 Kyiv, Ukraine,    \\
Email: iksan@univ.kiev.ua}   \\
\vspace{0.05cm}

\noindent
\textsc{Matthias Meiners    \\
Fachbereich Mathematik  \\
Technische Universit\"at Darmstadt  \\
64289 Darmstadt, Germany    \\
Email: meiners@mathematik.tu-darmstadt.de}
\end{document}